\documentclass{amsart}

\title[Blowing up the power of a singular cardinal with collapses]{Blowing up the power of a singular cardinal of uncountable cofinality with collapses}

\author{Sittinon Jirattikansakul}
\date{\today}
\usepackage{amssymb}
\usepackage{amsmath}
\usepackage{hyperref}
\usepackage[dvipsnames]{xcolor} 
\usepackage{enumitem}
\usepackage{tikz-cd}

\DeclareMathOperator{\dom}{dom}
\DeclareMathOperator{\rge}{rge}

\DeclareMathOperator{\cf}{cf}

\DeclareMathOperator{\OB}{OB}
\DeclareMathOperator{\ot}{ot}
\DeclareMathOperator{\mc}{mc}
\DeclareMathOperator{\ran}{ran}
\DeclareMathOperator{\supp}{supp}
\DeclareMathOperator{\Col}{Col}
\DeclareMathOperator{\Ult}{Ult}

\newtheorem{lemma}{Lemma}
\newtheorem{thm}{Theorem}

\newtheorem{claim}{Claim}[thm]
\theoremstyle{definition}
\newtheorem{defn}[thm]{Definition}
\newtheorem{rmk}[thm]{Remark}

\usepackage[dvipsnames]{xcolor}

\newenvironment{claimproof}[1]{\par\noindent\underline{Proof:}\space#1}{\hfill $\blacksquare$}

\begin{document}
\begin{abstract}

The {\em Singular Cardinal Hypothesis} (SCH) is one of the most classical combinatorial principles in set theory. It says that if $\kappa$ is singular strong limit,
  then $2^{\kappa}=\kappa^+$. We prove that given a singular cardinal $\kappa$ of {\em cofinality}  $\eta$ in the ground model,
  where $\kappa$ is a limit of suitable large cardinals, then there is a forcing extension which preserves cardinals and cofinalities up to and including $\eta$, such that if $\eta=\aleph_\gamma$, then $\kappa$ becomes $\aleph_{\gamma+\eta}$, and SCH fails at $\kappa$. 
  Furthermore, if $\eta$ is not an $\aleph$-fixed point, then in our model, SCH fails at $\aleph_{\eta}$.
  Our large cardinal assumption is the existence of a sequence of strong cardinals of length $\eta$. In our model we also obtain a very good scale.


\end{abstract}
\maketitle

\section{introduction}

 The {\em Singular Cardinal Hypothesis (SCH)} says, roughly speaking, that if $\theta$ is a singular cardinal then
 $2^\theta$ has the smallest value consistent with the restrictions provable in ZFC: 
in particular, SCH implies that if $\theta$ is a singular strong limit cardinal, then $2^\theta=\theta^+$.
Failure of SCH is known to require some large cardinal hypotheses.

The first models for the failure of SCH were built by starting with a large cardinal $\kappa$,
  forcing that GCH fails at $\kappa$ while preserving some large cardinal property of $\kappa$, and then
  making $\kappa$ singular via Prikry-type forcing. Gitik and Magidor \cite{SCHrevisited} gave a construction
  which starts with a large cardinal and simultaneously adds many cofinal $\omega$-sequences in $\kappa$, while
  keeping it strong limit.

In subsequent work, Gitik and Magidor \cite{EBF} gave a construction which starts with  
a cardinal $\kappa$ that is the limit of $\omega$ many strong cardinals, and adds many cofinal subsets
of $\kappa$ while keeping $\kappa$ strong limit. Recently, Gitik \cite{Gitikblow} showed how to achieve a similar result
when $\kappa$ is a limit of large cardinals and has uncountable cofinality. In the resulting extension $\kappa$
is quite large, for example, it is a cardinal fixed point. 

   We prove the following theorem:

\begin{thm}\label{main}
 Assume GCH holds, and there is an increasing sequence of cardinals
    $\langle \kappa_\alpha : \alpha < \eta \rangle$ and an integer $n \geq 2$ such that,
 letting  $\rho=\sup_{\alpha<\eta}\kappa_\alpha$ and  $\lambda=\rho^{+n}$: 
\begin{enumerate}
\item $0<\eta<\kappa_0$ and $\eta$ is a regular cardinal. 
\item There is  a sequence of $(\kappa_\alpha, \lambda)$-extenders $\langle E_\alpha : \alpha< \eta \rangle$
  such that: 
\begin{enumerate}
\item If $M_\alpha= \Ult(V, E_\alpha)$, then  $M_\alpha$ computes cardinals correctly up to and including $\lambda$,
  and ${}^{\kappa_\alpha} M_\alpha \subseteq M_\alpha$.
\item If $j_{E_\alpha} : V \rightarrow M_\alpha$ is the ultrapower map, then 
 there is $h_\alpha^\beta:\kappa_\alpha \to V_{\kappa_\alpha}$ such that $j_{E_\alpha}(h_\alpha^\beta)(\kappa_\alpha)=E_\beta$ for $\beta<\alpha$.
 \item There is a function $s_\alpha: \kappa_\alpha \to \kappa_\alpha$ with $j_{E_\alpha}(s_\alpha)(\kappa_\alpha)=\sup_{\alpha<\eta}\kappa_\alpha$.
\end{enumerate}
\end{enumerate}
Let $\eta^+=\aleph_{\gamma}$. Then there is a $\rho^{++}$-c.c.~forcing poset which preserves cardinals and cofinalities up to and including $\eta$, such that in the generic extension, $\sup\limits_{\alpha<\eta}\kappa_\alpha=\aleph_{\gamma+\eta}$, $2^{\aleph_{\gamma+\beta}} \geq \aleph_{\gamma+\beta}^{+n}$ for each limit ordinal $\beta < \eta$, and $2^{\aleph_{\gamma+\eta}}=\aleph_{\gamma+\eta}^{+n}$.  
\end{thm}


We make a few remarks on how to change the inequality in Theorem \ref{main}, and how to remove the noise ``$\gamma$" in some cases.  Notice that in Theorem \ref{main}, if $\eta$ is not an $\aleph$-fixed point, then $\gamma<\eta \leq \gamma+\eta$, so that in our model, SCH fails at $\aleph_\eta$. The condition ``$j_{E_\beta}(E_\alpha) \restriction \lambda = E_\alpha$ for $\alpha < \beta < \eta$'' implies that the sequence $\langle E_\alpha: \alpha<\eta \rangle$ is {\em Mitchell increasing}
  in the sense that $E_\alpha \in \Ult(V, E_\beta)$ for $\alpha < \beta < \eta$.
Gitik has conjectured that the existence of an uncountable Mitchell increasing sequence as above (with $n=2$) is optimal
to obtain failure of SCH at a singular cardinal of uncountable cofinality, which is singular in the core model.

Our forcing construction is a combination of Gitik's construction \cite{Gitikblow} which blows up the powerset of a singular cardinal with any cofinality,
            and his recent work on collpasing generators \cite{Gitikcol}.

We assume the reader is familiar with large cardinals and forcing.
Our presentation of extender-based forcing follows the same lines
as some work of  Merimovich \cite{Merimovichpri}.

Our notations are mostly standard. For each sequence of ordinal length
  $X = \langle x_i: i<\alpha \rangle$ and $\beta<\alpha$,
  we write $X \restriction \beta =\langle  x_i : i<\beta \rangle$,
  and $X \setminus \beta = \langle x_i : \beta \leq i <\alpha \rangle$.  
 
The organization of this paper is as follows:
  
 \begin{itemize} 
  \item In Section \ref{prelim}, we describe the assumptions we use to build the forcing. We then introduce the notions of domains and objects, explain the connection between objects and extenders, and do some extender analysis.

 \item In Section \ref{forcing}, we decsribe the forcing construction and explain how the extension works. After that, we describe the chain condition and closure of the forcing,
            and prove an {\em integration lemma}, which plays a vital role in proving the Prikry property.

  \item In Section \ref{prikry}, we prove the Prikry property and the strong Prikry property.

  \item In Section \ref{cardinalpreserve}, we describe the cardinals which are preserved.

  \item In Section \ref{scaleanalysis}, we analyze some scales. We use scales to show cardinal arithmetic, and also show that one of the scales derived from the forcing is very good.
  
  \end{itemize}
\section{preliminaries}\label{prelim}

  
We start with a ground model $V$ in which GCH holds.
  Our assumptions are slightly more general than those of Theorem \ref{main}, because we will establish some properties
  of the main forcing (notably the Prikry property) by induction on $\eta$. In our assumptions
  $\lambda$ is still of the form $\rho^{+n}$, but in our induction argument for the Prikry property, the proof requires extenders of longer lengths, hence we permit $\rho$ to be larger than $(\sup_{\alpha < \eta} \kappa_\alpha)^{+n}$. Throughout the paper, we treat the case $n=2$. One can modify our analysis to generalize for any larger $n$.

Fix an ordinal $\eta>0$ ($\eta$ can be finite), and 
a sequence of cardinals $\langle \kappa_\alpha : \alpha<\eta \rangle$ with  $\eta<\kappa_0$.
For each $\alpha$ with  $0< \alpha \leq \eta$,
let $\overline{\kappa}_\alpha = \sup_{\beta<\alpha} \kappa_\beta$,
and let $\overline{\kappa}_0 \in [\max\{\omega,\eta\},\kappa_0)$ be an ordinal.
Let $\rho \geq \sup_{\alpha<\eta} (\kappa_\alpha^{++})$, where $\rho$ is either an inaccessible cardinal or a limit of inaccessible cardinals,
and let $\lambda=\rho^{++}$. Note that the double successor operations appearing in a lower bound for $\rho$ is meaningful since $\eta$ can be successor.

Assume that\label{extenderassumption} :
\begin{enumerate}
    \item \label{closure,correctness}For each $\alpha$, $\kappa_\alpha$ carries a $(\kappa_\alpha,\lambda)$-extender $E_\alpha$. 
    \item \label{exincrease}Let  $j_{E_\alpha}:V \to M_\alpha$ be derived from the extender $E_\alpha$.
      Then $M_\alpha$ is closed under $\kappa_\alpha$-sequences, and $M_\alpha$ computes
      cardinals correctly up to and including $\lambda$.
    \item  \label{laverdia} There is a function $s_\alpha: \kappa_\alpha \to \kappa_\alpha$
      with $j_{E_\alpha}(s_\alpha)(\kappa_\alpha)=\rho$, and $s_\alpha(\nu) > \max\{ \nu,\overline{\kappa}_\alpha \}$ for all $\nu$.
 
    \item  \label{mitchell} For $\beta<\alpha$, there is a function $h_\alpha^\beta:\kappa_\alpha \to V_{\kappa_\alpha}$ such that $j_{E_\alpha}(h_\alpha^\beta)(\kappa_\alpha)=E_\beta$.
 \end{enumerate}   
As we noted already, the condition \ref{mitchell} implies the {\em Mitchell increasing property}, namely,
for $\alpha<\beta<\eta$, $E_\alpha \in M_\beta$.

Observe that the sequence $\langle \kappa_\alpha : \alpha<\eta \rangle$ is \textbf{not} continuous.
In general, for $\alpha$ with $\alpha+1\leq \eta$,
$\overline{\kappa}_{\alpha+1}=\kappa_\alpha$. In addition, for each $\alpha<\eta$, $\overline{\kappa}_\alpha<\kappa_\alpha$, hence,
requirement \ref{laverdia} is sensible.


The following representation of the extenders is due to Merimovich \cite{Merimovichpri}.
For each $\alpha<\eta$, we let an $\alpha$-{\em domain} be a set $d\in[\lambda]^{\kappa_\alpha}$
such that $\kappa_\alpha+1 \subseteq d$. In the original Merimovich context $\kappa_\alpha=\min(d)$, but it our case, it is more convenient to have $d \supseteq \kappa_\alpha+1$. Define $\mc_\alpha(d)$ as $(j_{E_\alpha} \restriction d)^{-1}$.
We represent $E_\alpha= \langle E_\alpha(d) : \mbox{$d$ is an $\alpha$-domain} \rangle$, where
$X \in E_\alpha(d)$ iff $\mc_\alpha(d) \in j_{E_\alpha}(X)$.
From this point, if there is no ambiguity about the value of $\alpha$,
we may drop $\alpha$ from the notation (refer to an $\alpha$-domain as a domain,
$E_\alpha$ as $E$, $\mc_\alpha(d)$ as $\mc(d)$ and so on).  We define a set on which $E_\alpha(d)$ concentrates.

\begin{defn}
Let $d$ be an $\alpha$-domain. $\OB_\alpha(d)$ is the collection of $d$-$\alpha$-{\em objects}, which are the functions $\mu$ such that 
\begin{enumerate}
    \item $\dom(\mu) \subseteq d$, $\ran(\mu) \subseteq \kappa_\alpha$, and $\kappa_\alpha \in \dom(\mu)$.
    \item $|\dom(\mu)| = \mu(\kappa_\alpha)$ (which is below $\kappa_\alpha)$, and $\mu(\kappa_\alpha)$ is inaccessible.
    \item $\dom(\mu)\cap \kappa_\alpha=\mu(\kappa_\alpha)$.
    \item $\mu$ is order-preserving.
    \item for $\beta \in \dom(\mu)\cap \kappa_\alpha$, $\mu(\beta)=\beta$.

\end{enumerate}

\end{defn}

As mentioned earlier, it is straightforward to check that $E_\alpha(d)$ concentrates on $d$-$\alpha$-objects.
It is easy to see that for each domain $d$, there is a unique $\alpha<\eta$ such that $d$ is an $\alpha$-domain.
Hence, we use the term ``$d$-object'' to refer to a $d$-$\alpha$-object for the unique $\alpha$
   such that $d$ is an $\alpha$-domain.
The collection $\langle E_\alpha(d) : \textrm{$d$ is an $\alpha$-domain} \rangle$ comes with natural projections.
For each pair of $\alpha$-domains $d \subseteq d^\prime$ define $\pi_{d^\prime,d}: \OB_\alpha(d^\prime) \rightarrow \OB_\alpha(d)$
by $\pi_{d^\prime,d}(\mu)=\mu \restriction d$: it is routine to check that $\pi_{d^\prime, d}$ projects $E_\alpha(d')$ to $E_\alpha(d)$. We do some analysis of extenders in Lemma \ref{rowbottom} and Lemma \ref{sparseuf}.

\begin{lemma}\label{rowbottom}
Suppose $0 \leq \alpha_0 < \dots < \alpha_{k-1}< \eta$, for each $j< k$, $d_i$ is an $\alpha_i$-domain, $A_i \in E_{\alpha_i}(d_i)$, and $F:\prod_{i<k}A_i \to 2$. Then there are $B_i \subseteq A_i$, $B_i \in E_{\alpha_i}(d_i)$ such that $F \restriction (\prod_{i<k} B_i)$ is constant.
\end{lemma}

\begin{proof}

  Induct on $k$. The case $k=1$ is trivial. 
  Suppose $k>1$. For each $\mu \in A_{\alpha_{k-1}}$, let $F_\mu: \prod_{i<{k-1}}A_i \to 2$ be defined by $F_\mu(\vec{x})=F(\vec{x}^\frown  \langle \mu \rangle)$. 
  Since $\kappa_{\alpha_{k-1}}$ is inaccessible, by completeness of $E_{\alpha_{k-1}}(d_{k-1})$, there exist $B_{k-1} \subseteq A_{k-1}$ and is measure one for the appropriate measure, and $F'$ such that $F_\mu = F'$ for all $\mu \in B_{k-1}$.
   Now apply the induction hypothesis to find $B_i \subseteq A_i$, $B_i \in E_{\alpha_i}(d_i)$ for $i<k-1$ with $F^\prime \restriction (\prod_{i<k-1}B_i)$ constant. 
   It is easy to see that $F \restriction (\prod_{i<k}B_i)$ is constant.
  
\end{proof}

Lemma \ref{sparseuf} below is a pointwise reflection of the size of the extender.
\begin{lemma}\label{sparseuf}
  For each $\alpha<\eta$ and each $\alpha$-domain $d$, there is a set $A_\alpha(d)$ such that
  $A_\alpha(d) \in E_\alpha(d)$, and for each $\nu<\kappa_\alpha$,
  the size of $\{\mu \in A_\alpha(d) : \mu(\kappa_\alpha)=\nu\}$ is at most $s_\alpha(\nu)^{++}$.
\end{lemma}

\begin{proof}

  We drop the subscript $\alpha$ for this proof and fix a domain $d$.
  Let $d^*=d \setminus \kappa$. For each $\mu \in \OB(d)$, set $\mu^*=\mu \restriction (\dom(\mu) \setminus \kappa)$,
  and let  $\mc^*(d)=\mc(d) \restriction (\dom(\mc(d)) \setminus \kappa)$.
  Enumerate $d^*$ as $\langle \delta_i : i<\kappa \rangle$.
  Let $B_d$ be the set of $\mu \in \OB(d)$ such that $\dom(\mu^*)=\{\delta_i : i<\mu(\kappa)\}$. We claim that $B_d \in E(d)$.
  The point is that $\dom(\mc^*(d))=j[d^*]=\{j(\delta)_i : i<\kappa\}$, $\mc^*(d)(j(\kappa))=\kappa$,  and
  $j( \langle \delta_i : i<\kappa \rangle) \restriction \kappa= \langle j(\delta_i) : i<\kappa \rangle$.

Let $\vec{t}= \langle t_\delta  : \delta<\kappa \rangle$ be an enumeration of $[\kappa]^{<\kappa}$ such that whenever $\delta<\kappa$ is a closure point of the map $\gamma \mapsto s(\gamma)^{++}$, $\langle t_\beta : \beta \in [\delta,s(\delta)^{++}) \rangle$ enumerates the set of  $t \in [s(\delta)^{++}]^{\leq \delta}$ with $\min(t)=\delta$. 
Let $j(\vec{t}) = \vec{T}$. Since $\kappa$ is a closure point for $j(\gamma \mapsto s(\gamma)^{++})$, and $j(s)(\kappa)^{++}=\lambda$, we have that 
  $\langle T_\delta \sqcup \kappa : \kappa \leq \delta < \lambda \rangle$ enumerates all domains.
 Choose $\delta$ so that $T_\delta = d^*$. For each $\beta \in d^*$, let $i = \ot(d^* \cap \beta) < \kappa^+$,
  and let $\pi_{\beta}$ be the function which takes $\nu < \kappa$ to the $f_i(\min(t_\nu))^{\rm th}$ element
  of $t_\nu$ where $f_i$ is the $i^{\rm th}$ canonical function. Then $j(\pi_{\beta})(\delta)$ is the
  $j(f_i)(\kappa)^{\rm th}$ element of $T_\delta$, which is $\beta$. 

  Define
  \[
  A_d = \{ \mu \in B_d : \exists \xi < s(\mu(\kappa))^{++} \; \forall \beta \in \dom(\mu^*) ,
  \mu^*(\beta) = \pi_\beta(\xi) \}
  \]

  We claim that $A_d \in E(d)$. Note that $\delta<\lambda=(j(s)(\mc(d)(j(\kappa))))^{++}$ and for each $\beta \in d^*$, $\mc^*(d)(j(\beta))=\beta=j(\pi_\beta)(\delta)$.

  The conclusion of the proof follows from the fact that given
  $\tau \in A_d$ with $\tau(\kappa)=\nu$, $\dom(\tau)= \nu \cup \{ \delta_i: i<\nu\}$, $\tau(\beta)=\beta$ for all $\beta < \nu$,
  and there is $\xi<s(\nu)^{++}$ such that for $\beta \in \dom(\mu) \setminus \kappa$ , $\mu(\beta)=\pi_\beta(\xi)$.

  Hence, $\nu$ and $\xi$ completely determine $\tau$.
 For every $\nu$ there are only $s(\nu)^{++}$ possible values for $\xi$. Hence, $A_d$ works as required.
\end{proof}

In the sequel we define $A_\alpha(d)$ to be the measure one set from the conclusion of Lemma \ref{sparseuf}.
If $\beta<\alpha$, then there is a notion of domain $d$ with respect to $h_\alpha^\beta(\gamma)$ for appropiate $\gamma$.
We will use the same notation $A_\alpha(d)$ in Lemma \ref{sparseuf}, as the particular measure-one set in $h_\alpha^\beta(\gamma)(d)$.

\section{The  forcing}\label{forcing}

The forcing is constructed using ideas from Gitik's recent preprints \cite{Gitikblow} and  \cite{Gitikcol}. To apply the induction hypothesis, in this section, we assume that $\overline{\kappa}_0$ is regular and $\max\{\omega,\eta\} \leq \overline{\kappa}_0< \kappa_0$.
Recall that $\rho \geq \sup_{\alpha<\eta}(\kappa_\alpha)^{++}$ and $\lambda=\rho^{++}$. We introduce some notation here. For each pair of cardinals $\kappa$ and $\theta$,
  let $\mathcal{A}(\kappa, \theta)$ be the set of partial functions from $\theta$ to $\kappa$
  with domains which have size $\kappa$ and contain $\kappa+1$,
  ordered by extension. 
  Also, define $E_\alpha(\kappa_\alpha)$ by $X \in E_\alpha(\kappa_\alpha)$
  iff $\kappa_\alpha \in j_{E_\alpha}(X)$. This is just a normal measure on $\kappa_\alpha$.
  If $d$ is an $\alpha$-domain and $A \in E_\alpha(d)$ then
    we define $A(\kappa_\alpha) = \{ \mu(\kappa_\alpha) : \mu \in A \}$,
    and note that easily $A(\kappa_\alpha) \in E_\alpha(\kappa_\alpha)$. Assume all the hypotheses from the beginning of Section \ref{prelim}.

\begin{defn}\label{poset}
  The forcing $\mathbb{P}_{\langle E_\alpha : \alpha<\eta\rangle}$ consists of sequences
  $p= \langle p_\alpha : \alpha<\eta \rangle$ such that for  a finite subset $\supp(p)$ of $\eta$ (the {\em support of $p$}):
\begin{align*}
p_\alpha=
\begin{cases}
\langle f_\alpha,\lambda_\alpha,h^0_\alpha,h^1_\alpha,h^2_\alpha \rangle & \text{if }  \alpha \in \supp(p), \\
\langle f_\alpha,A_\alpha,H^0_\alpha,H^1_\alpha,H^2_\alpha \rangle & \text{otherwise},
\end{cases} 
\end{align*}
where:

  \begin{enumerate}

    \item \label{cards} $\overline{\kappa}_\alpha < \lambda_\alpha < \kappa_\alpha$ for all $\alpha \in\supp(p)$.

    \item\label{alphastar}  $f_\alpha \in \mathcal{A}(\kappa_\alpha,\lambda_{\alpha^*})$ where $\alpha^*=\min(\supp(p) \setminus (\alpha+1))$ if it exists, otherwise $\lambda_{\alpha^*}=\lambda$.
   We denote $\dom(f_\alpha)$ by $d_\alpha$, and note that $d_\alpha$ is an $\alpha$-domain.
  
\item \label{domin} $\{d_\alpha : \alpha < \eta \}$ is $\subseteq$-increasing.

\item \label{reflect}  If $\alpha \in \supp(p)$ then:
  \begin{itemize}
  \item $f_\alpha(\kappa_\alpha)$ is inaccessible and $f_\alpha(\kappa_\alpha)>\overline{\kappa}_\alpha$. 
  \item  $\lambda_\alpha=\rho_\alpha^{++}$ where $\rho_\alpha = s_\alpha(f_\alpha(\kappa_\alpha))$, so that
    in particular $\overline{\kappa}_\alpha < \overline{\kappa}_\alpha^+ <  f_\alpha(\kappa_\alpha) < \rho_\alpha < \rho_\alpha^+ <\lambda_\alpha<
      \lambda_\alpha^+ < \kappa_\alpha$.

  \end{itemize}

\item \label{measureone} If $\alpha \not \in \supp(p)$, then for $\alpha>\max(\supp(p))$, $A_\alpha \in E_\alpha(d_\alpha)$, and for $\alpha<\max(\supp(p))$, $A_\alpha \in h_{\alpha^*}^\alpha(f_{\alpha^*}(\kappa_{\alpha^*}))$, where $\alpha^*$ is defined as in \ref{alphastar}. Furthermore, $A_\alpha \subseteq A_\alpha(d_\alpha)$.

\item \label{impurecoll} If $\alpha \in \supp(p)$, then $h_\alpha^0 \in \Col(\overline{\kappa}_\alpha^+,<f_\alpha(\kappa_\alpha)), h_\alpha^1 \in \Col(f_\alpha(\kappa_\alpha),\rho_\alpha^+)$, and $h_\alpha^2 \in \Col(\lambda_\alpha^+,<\kappa_\alpha)$.

\item \label{purecoll} If $\alpha \not \in \supp(p)$, then 
\begin{enumerate}

\item\label{H0H1dom}
  $\dom(H_\alpha^0)=\dom(H_\alpha^1) = A_\alpha$.

\item\label{H0H1coll}
For $\tau \in \dom(H_\alpha^0)$ let $\nu=\tau(\kappa_\alpha)$, then $H_\alpha^0(\tau)
   \in \Col(\overline{\kappa}_\alpha^+,<\nu)$, and $H_\alpha^1(\tau) \in \Col(\nu,s_\alpha(\nu)^+)$.

\item \label{H2dom}
 $\dom(H_\alpha^2) = A_\alpha(\kappa_\alpha)$. 

\item \label{H2coll} For each ordinal $\nu \in \dom(H_\alpha^2)$, $H_\alpha^2(\nu) \in \Col(s_\alpha(\nu)^{+3},<\kappa_\alpha)$.

\end{enumerate}

\end{enumerate}

  \end{defn}

Readers of Gitik's preprint \cite{Gitikcol} will notice a few differences in the collapsing parts, in particular $H_\alpha^0$ and $H_\alpha^1$ are defined with domains in $E_\alpha(d_\alpha)$ rather than a projected measure. This is not necessary here because the pattern of collapses is a bit different,
  the projection in \cite{Gitikcol} is needed for the chain condition, but this is not necessary here and hence, permits
  our conditions to be simpler.

There are two main reasons why our situation is simpler than that in \cite{Gitikcol}. The first one is that as in
    \cite{Gitikblow} the ``impure'' part of a condition is rather small. The other is that we are defining the forcing
    using Merimovich's machinery of objects, which simplifies the ultrapower analysis. Both of these factors play a role
    in the chain condition analysis in  Theorem \ref{chaincond} below.

To clarify that all the requirements are reasonable, fix an $\alpha<\eta$.
 If $\alpha^*$ exists as in \ref{alphastar}, then $\kappa_\alpha \leq \overline{\kappa}_{\alpha^*}<\lambda_{\alpha^*}<\kappa_{\alpha^*}$.
  The forcing $\mathcal{A}(\kappa_\alpha,\lambda_{\alpha^*})$ is equivalent to the forcing adding $\lambda_{\alpha^*}$-subsets of $\kappa_\alpha^+$.
 
  The fact that for $\alpha \in \supp(p)$,
  $f_\alpha(\kappa_\alpha)$ is inaccessible will eventually guarantee that the forcing preserves finitely many cardinals at each coordinate,
  see Lemma \ref{preserve}.

When we discuss multiple conditions, we usually put a superscript on every component
with the name of the condition. Also, a condition is officially a sequence of sequences, but
we sometimes drop the angle brackets if it is clear from the context.
For example, when $\eta=1$, we may write $p=\langle f_0^p, A_0^p,(H_0^0)^p,(H_0^1)^p,(H_0^2)^p \rangle$
instead of having two pairs of brackets.
As one might expect, objects in $A_\alpha$ are used to extend a condition in a meaningful way. We will restrict our attention to the valid objects (which we call {\em addable}) in Definition \ref{addability}. 

We give some conventions here. The {\em pure part} of $p$ is the $p_\alpha$'s for $\alpha \not \in \supp(p)$. The rest is called the {\em impure part} of $p$. We refer the first coordinate $f_\alpha$ of
each $p_\alpha$ as the {\em Cohen part}, and the last three coordinates of
each $p_\alpha$ as the {\em collapse parts}. A condition $p$ is {\em pure} if $\supp(p)=\emptyset$. For each $\alpha$-object $\mu$, let $\rho_\alpha(\mu)=s_\alpha(\mu(\kappa_\alpha))$, $\lambda_\alpha(\mu)=\rho_\alpha(\mu)^{++}$.
For each $\nu<\kappa_\alpha$, let $\rho_\alpha(\nu)=s_\alpha(\nu)$, and $\lambda_\alpha(\nu)=\rho_\alpha(\nu)^{++}$.


As usual, we drop the subscript $\alpha$ if it is clear from the context. Similarly, we sometimes drop the subscript $\langle E_\alpha : \alpha <\eta \rangle$ from $\mathbb{P}_{\langle E_\alpha : \alpha <\eta \rangle}$. We define
$\mathbb{P} \restriction \alpha$ to be $\{p \restriction \alpha : p \in \mathbb{P}\}$,
and $\mathbb{P} \setminus \alpha$ to be $\{p \setminus \alpha : p \in \mathbb{P}\}$.
The notions in Definition \ref{poset} refer to the corresponding components. For instance, $f_0^p$ is the Cohen part of $p_0$, the first entry of $p$, and $d_0^p$ is just $\dom(f_0^p)$.
  
We now define the concept of {\em direct extension}.

\begin{defn}
Let $p,q \in \mathbb{P}$. $p$ is a \textbf{direct extension} of $q$, denote $p \leq^* q$, if and only if:

\begin{enumerate}
    \item $\supp(p)=\supp(q)$.
    \item For each $\alpha$, $f_\alpha^p \leq f_\alpha^q$.
    \item If $\alpha \in \supp(p)$, then $\lambda_\alpha^p=\lambda_\alpha^q$.
    \item If $\alpha \in \supp(p)$, then $(h_\alpha^l)^p \leq (h_\alpha^l)^q$ for $l < 3$.
    \item If $\alpha \not \in \supp(p)$, then 
    \begin{enumerate}
    \item \label{projects} $\pi_{d_\alpha^p,d_\alpha^q}[A_\alpha^p] \subseteq A_\alpha^q$ (recall that $\pi_{d_\alpha^p,d_\alpha^q}(\tau)=\tau \restriction d_\alpha^q$) . As a consequence
      $\pi_{d_\alpha^p,d_\alpha^q}[\dom(H_\alpha^l)^p]  \subseteq \dom(H_\alpha^l)^q$ for $l < 2$ and $\dom(H_\alpha^2)^p \subseteq \dom(H_\alpha^2)^q $.
      
        \item For $l<2$, $(H_\alpha^l)^p(\tau) \leq (H_\alpha^l)^q(\pi_{d_\alpha^p,d_\alpha^q}(\tau))$ for all relevant $\tau$.

        \item For all relevant $\nu$, $(H_\alpha^2)^p(\nu) \leq (H_\alpha^2)^q(\nu)$.
          
    \end{enumerate}
\end{enumerate}
\end{defn}

We often express property \ref{projects} as ``$A_\alpha^p$ projects down to a subset of $A_\alpha^q$". To see what
the extensions of a given condition look like in general, we first restrict the kind of
  objects which are allowed to be used to extend a condition.


  The following definition and other definitions later on involves composition of functions. We note here that compositions are done partially, meaning for functions $f$ and $g$, $f \circ g$ has the domain $\{x \in \dom(g) : g(x) \in \dom(f)\}$ and $f \circ g(x)=f(g(x))$.

\begin{defn}\label{addability}
  Let $p \in \mathbb{P}$, $\alpha \not \in \supp(p)$, and  $\mu \in A_\alpha^p$. $\mu$ is \textbf{addable} to $p$
    if:
\begin{enumerate}
    \item $\overline{\kappa}_\alpha<\mu(\kappa_\alpha)$, and $\mu(\kappa_\alpha)$ is inaccessible.
    
    \item $\bigcup\limits_{\beta<\alpha} d_\beta \subseteq \dom(\mu)$ and $\mu \restriction (\overline{\kappa}_\alpha+1)$ is an identity function.

  \item For each  $\beta \in (\max(\supp(p)\cap \alpha),\alpha)$,
    
    \begin{enumerate}
    
    \item $\mu[d_\beta] \subseteq \lambda_\alpha(\mu)$ (recall that $\lambda_\alpha(\mu)=s_\alpha(\mu(\kappa))^{++})$.
    
    \item \label{composewithmuinverse} $\{\psi \circ \mu^{-1} : \psi \in A_\beta\} \in h_\alpha^\beta(\mu(\kappa_\alpha))(\mu[d_\beta])$.
    
    \item $\mu(\kappa_\beta)=\kappa_\beta$.
        
        \item $h_\alpha^\beta(\mu(\kappa_\alpha))$ is a $(\kappa_\beta,s_\alpha(\mu(\kappa_\alpha)))$-extender.

\end{enumerate}

\item for $\gamma<\beta<\alpha$, $j_{h_\alpha^\beta(\mu(\kappa_\alpha))}(h_\beta^\gamma)(\kappa_\beta)=h_\alpha^\gamma(\mu(\kappa_\alpha))$.

\end{enumerate}
\end{defn}

We denote the set in condition \ref{composewithmuinverse} as $A_\beta \circ \mu^{-1}$.
It is important that almost every $\mu \in A_\alpha$ is addable, and the proof illustrates the role of coherence, so we sketch it.
We write $\mc_\alpha$ for $\mc_\alpha(d_\alpha)$.  
\begin{itemize}
\item  $j_{E_\alpha}(\overline{\kappa}_\alpha) = {\overline \kappa}_\alpha < \kappa_\alpha = \mc_\alpha(j_{E_\alpha}(\kappa_\alpha))$, and
  $\kappa_\alpha$ is inaccessible in $M_\alpha$.
\item If $D = \bigcup_{\beta < \alpha} d_\beta$ then $\vert D \vert \le \overline{\kappa}_\alpha$ and $D \subseteq d_\alpha$, so
  $j_{E_\alpha}(D) = j_{E_\alpha}[D] \subseteq \dom(\mc_\alpha)$.
\item  By the previous item, for $\beta<\alpha$, $j_{E_\alpha}(d_\beta) \subseteq \dom(\mc_\alpha)$, and
  easily $\rge(\mc_\alpha) \subseteq \lambda$. By the choice of $s_\alpha$,
   $j_{E_\alpha}(s_\alpha)(\mc_\alpha(j_{E_\alpha}(\kappa_\alpha))) = \rho$. Now $\lambda = \rho^{++}$
   and so easily $j_{E_\alpha}(\lambda_\alpha)(\mc_\alpha(j_{E_\alpha}(\kappa_\alpha))) = \lambda$.
\item Note that for each $\psi \in A_\beta$, $\dom(\psi) \subseteq d_\beta \subseteq \dom(\mu)$, so that
  $\dom(\psi \circ \mu^{-1}) = \mu[\dom(\psi)]$. By a routine calculation,
  $\mc_{\alpha}[j_{E_\alpha}(\dom(\psi))] = \dom(\psi)$, and $j_{E_\alpha}(\psi) \circ \mc_\alpha^{-1}$ maps
  $\gamma \in \dom(\psi)$ to $\psi(\gamma)$, that is $j_{E_\alpha}(\psi) \circ \mc_\alpha^{-1} = \psi$.
  Since $\vert A_\beta \vert = 2^{\kappa_\beta} < \kappa_\alpha$, $j_{E_\alpha}(A_\beta) = j_{E_\alpha}[A_\beta]$ and so
  $\{ \Psi \circ \mc_\alpha^{-1} : \Psi \in j_{E_\alpha}(A_\beta) \} = A_\beta$.
  Now $A_\beta \in E_\beta(d_\beta) = j_{E_\alpha}(h_\alpha^\beta)(\kappa_\alpha)(d_\beta)$ where $E_\beta$ is a $(\kappa_\beta,j_{E_\alpha}(s_\alpha)(\kappa_\alpha))$-extender.
  \item If $\gamma<\beta<\alpha$, then $E_\gamma=j_{E_\beta}(h_\beta^\gamma)(\kappa_\beta)=j_{E_\alpha}(\xi \mapsto j_{h_\alpha^\beta(\xi)}(h_\beta^\gamma)(\kappa_\beta))(\kappa_\alpha)$ and $E_\gamma=h_\alpha^\gamma(\mu(\kappa_\alpha))$.
\end{itemize}

It will become clearer why we restrict ourselves to the addable objects, once we define the notion of the one-step extension
in Definition \ref{onestep}.
To define a non-direct extension, we introduce more notation.
    If $f \in \mathcal{A}(\kappa_\alpha,\lambda)$ with domain $d$,
    and $\mu$ is an $\alpha$-$d$-object, we define $f \oplus \mu$ to be the
    function $g \in \mathcal{A}(\kappa_\alpha,\lambda)$ such that $\dom(g)=\dom(f)$, and
 \begin{align*}
    g (\gamma)=
\begin{cases}
\mu(\gamma) & \text{if }  \gamma \in \text{dom}(\mu), \\
f(\gamma) & \mbox{otherwise}. 
\end{cases}
\end{align*}  

Note that we obtain the function $g$ by simply overwrite previous values by $\mu$. For each condition $p$, $\beta<\alpha$, and $\mu \in E_\alpha(d_\alpha^p)$, we define $(f_\beta^p)_\mu$ as $f_\beta^p \circ \mu^{-1}$. For $l<2$ we define $(H_\beta^l)^p_\mu$ with domain $A_\beta^p \circ \mu^{-1}$, and $(H_\beta^l)^p_\mu(\xi^\prime)=(H_\beta^l)^p(\xi^\prime \circ \mu)$.

 \begin{defn}\label{onestep}(one-step extension) 
 Fix a condition $p \in \mathbb{P}$, $\alpha \not \in \supp(p)$, and an addable object $\mu \in A_\alpha^p$. The \textbf{one-step extension} of $p$ by $\mu$, denoted by $p+\mu$ is the condition $q$, where

\begin{enumerate}
    \item $\supp(q)=\supp(p) \cup \{\alpha\} $.
    \item $q \restriction [0,\max(\supp(p)\cap \alpha))=p \restriction [0,\max(\supp(p)\cap \alpha))$, and $q\setminus \alpha=p\setminus \alpha$.
    \item At the $\alpha$-th coordinate, we have
    \begin{enumerate}
        \item $f^q_\alpha=f^p_\alpha \oplus \mu$.
    \item  $\lambda_\alpha^q=\lambda_\alpha(\mu)$.
    \item for $l<2$, $(h_\alpha^l)^q=(H_\alpha^l)^p(\mu)$.
    \item $(h_\alpha^2)^q=(H_\alpha^2)^p(\mu(\kappa_\alpha))$.
    \end{enumerate}
    \item \label{squishblock} fix $\beta \in [\max(\supp(p)\cap \alpha),\alpha)$. Then
    \begin{enumerate}
        \item $f_\beta^q=(f_\beta^p)_\mu$.
        \item $A_\beta^q=A_\beta^p \circ \mu^{-1}$.
        \item For $l<2$, $(H_\beta^l)^q=(H_\beta^l)^p_\mu$.
        \item $(H_\beta^2)^q=(H_\beta^2)^p$.
    \end{enumerate}
\end{enumerate}
\end{defn}

Let's warm up and get familiar with our notations by showing that if $p$ is pure and $\mu \in A_\alpha^p$ is addable, then $q:=p+\mu$ is indeed a condition, as in Definition \ref{poset}.
 

\begin{itemize}
\item (\ref{cards}) $\supp(q)=\{\alpha\}$. $\lambda_\alpha^q=\lambda_\alpha(\mu)=s_\alpha(\mu(\kappa_\alpha))^{++}<\kappa_\alpha$. By our assumption of $s_\alpha$, $s_\alpha(\mu(\kappa_\alpha))^{++}>s_\alpha(\mu(\kappa_\alpha))>\overline{\kappa}_\alpha$.
\item (\ref{alphastar}),(\ref{domin}) It is easy to see that for $\beta \geq \alpha$, $\lambda_{\beta^*}=\lambda$. The part above $\alpha$ is not affected, and $\dom(f_\alpha^p \oplus \mu)=\dom(f_\alpha^p)$. Hence, $\dom(f_\beta^q)=\dom(f_\beta^p)$. For $\beta<\alpha$, $\beta^*=\alpha$, and $\dom(f_\beta^q)=\dom(f_\beta^p \circ \mu^{-1})=\mu[d_\beta^p]\subseteq \lambda_\alpha(\mu)=\lambda_\alpha^q$. Since $\mu$ fixes ordinals below $\kappa_\beta+1$, $d_\beta^q$ is a $\beta$-domain.  Also, $\lambda_\alpha(\mu)<\kappa_\alpha+1 \subseteq \dom (f_\alpha^q)$. Hence, the domains in $q$ are $\subseteq$-increasing.
\item (\ref{reflect}) $f_\alpha^q(\kappa_\alpha)=\mu(\kappa_\alpha)>\overline{\kappa}_\alpha$, and $\mu(\kappa_\alpha)$ is inaccessible. The rest follows by our definitions.
\end{itemize}

The rest of the proof is trivial for $\beta>\alpha$, and we assume $\beta<\alpha$.

\begin{itemize}
\item (\ref{measureone}) Follows from addability of $\mu$.
\item (\ref{impurecoll}) Follows from (\ref{purecoll}) in Definition \ref{poset}.
\item (\ref{purecoll}) For $l=0,1$, $\dom(H_\beta^l)^q=\dom(H_\beta^l)^p_\mu=A_\beta^p \circ \mu^{-1}=A_\beta^q$. Since $d_\beta^p \subseteq \dom(\mu)$, for $\psi \in A_\beta ^p$, and $l=0,1$, $\dom(\psi) \subseteq d_\beta^p \subseteq \dom(\mu)$, so  $(H_\beta^l)^q(\psi \circ \mu^{-1})=(H_\beta^l)^p(\psi \circ \mu^{-1} \circ \mu)=(H_\beta^l)^p(\psi)$. $\dom(H_\beta^2)^p=A_\beta^p(\kappa_\beta)=A_\beta^p \circ \mu^{-1}(\kappa_\beta)$ since $\mu$ fixes $\kappa_\beta$. Also,  all the collapses fall into the right types, since $\mu$ fixes $\kappa_\beta$ as well.
\end{itemize}

   
 From Definition \ref{onestep}, for each $\beta<\alpha$, we sometimes write $q_\beta$ as $(p_\beta)_\mu$. When we perform a one-step extension of $p$ by $\mu$ from the $\alpha$-th coordinate, we call the part of the resulting condition before $\alpha$ ``$p \restriction \alpha$ {\em squished by} $\mu$". Occasionally, it is possible that we squish $p \restriction \alpha $ by some $\tau \in E_\alpha(d^\prime)$ for some $d^\prime \supseteq d$, where $d=\dom(f_\alpha^p)$. In this case $(p_\beta)_\tau=(p_\beta)_{\tau \restriction d}$.  For each sequence of objects $\vec{\mu}= \langle \mu_0,\mu_1,\dots,\mu_{k-1} \rangle$, we define $p+ \vec{\mu} $ recursively as
 $(p + \mu_0) + \langle \mu_1,\dots,\mu_{k-1} \rangle$ when $\mu_0$ is addable to $p$ and for every $0<i<k$,
 $\mu_i$ is addable to $p+\langle \mu_0, \dots, \mu_{i-1} \rangle$.

We define $p \leq q$ if $p \leq^* q+ \vec{\mu}$ for some sequence of objects $\vec{\mu}$. We refer the readers to Lemma 2.4 in \cite{ben2019} for some straightforward
technical lemmas on transitivity and commutativity of the choices of the orders of the objects we add to a condition. Briefly speaking, we ensure that the ordering $\leq$ is transitive.


For each condition $p \in \mathbb{P}_{\langle E_\alpha : \alpha<\eta \rangle}$ with $\alpha \in \supp(p)$, we can see that $\mathbb{P} / p \cong (\mathbb{P}/p) \restriction \alpha \times (\mathbb{P} / p) \setminus \alpha$, where the first factor can be regarded as
$(\mathbb{P} /p)_{\langle h_\alpha^\beta(f_\alpha^p(\kappa_\alpha)) : \beta<\alpha \rangle}$. The second factor can be regarded as
  $(\mathbb{P} /p \setminus \alpha)_{\langle E_\beta : \alpha \leq \beta<\eta \rangle}$.
  These two factors are independent of each other, despite the connection between $d_\alpha$'s on different levels
    in clause \ref{domin} of Definition \ref{poset}.
    The point is that $d_\alpha^p \supseteq \kappa_\alpha+1$ and $\lambda_\alpha<\kappa_\alpha$.
    Hence, any kinds of extensions in the first factor
    do not harm  property \ref{domin} in Definition \ref{poset}.

We begin our analysis of the poset by computing
    its chain condition.
    Note that even though we change the value $\lambda=\rho^{+n}$ for some $n>2$ as stated in Theorem \ref{main}, the forcing is still be $\rho^{++}$-c.c.


\begin{lemma}\label{chaincond}
If $\eta$ is a limit ordinal, then $\mathbb{P}_{\langle E_\alpha : \alpha<\eta \rangle}$ has the $\rho^{++}$-chain condition.
\end{lemma}

\begin{proof}
  Suppose the conditions $\langle p^i : i<\lambda \rangle$ are given. The first step is to show that without loss of generality, we can assume $p^i$ is pure for all $i$. Since $|[\eta]^{<\omega}|<\lambda$, we assume every condition has the same support $s$.
  Suppose $s \neq \emptyset$, and let $\alpha=\max(s)$.
 
  Each condition breaks into two pieces, with the first piece in $\mathbb{P}_{\langle h_\alpha^\beta(f_\alpha^{p^i}(\kappa_\alpha)): \beta<\alpha \rangle}$ Since the number of the conditions in the first factor is at most $\kappa_\alpha<\lambda$, we assume that every condition's first factor is the same, also, every condition has the same $\lambda_\alpha$. Now, for each $i<\lambda$, $\dom(f_\alpha^{p^i}) \in [\lambda]^{\kappa_\alpha}$. We thin out the collection $\{p^i:i<\lambda\}$ so that $\{\dom(f_\alpha^{p^i}) : i<\lambda\}$ form a $\Delta$-system. Then thin out the collection again so that the $\alpha$-th Cohen parts are all compatible. Since all the $\alpha$-th collapse parts are small, by the Pigeonhole Principle  we can also thin out the collection so that every condition has the same collapses at the $\alpha$-th coordinate. Hence, the initial segments up to and including $\alpha$ are compatible.

  Since we can shrink the collection $\{p^i:i<\lambda\}$ so that the impure parts are compatible, we now may assume all
  the conditions  $p^i$ are pure. Since $\vert \bigcup_{i < \eta} \dom(f_\alpha^{p^i}) \vert \le \overline{\kappa}_\eta \le \rho$,
    by a similar $\Delta$-system argument, we can thin out so that the
    Cohen conditions $f_\alpha^{p^i}$ are pairwise compatible in ${\mathcal A}(\kappa_\alpha, \lambda)$
    for each $\alpha<\eta$. Measure one sets are always compatible. It remains to find a subcollection of $\{p^i:i<\lambda\}$
 so that all the collapse parts $(H_\alpha^l)^{p^i}$ are compatible. Note that the collapses are no longer small.

For $l=0$ and $l=1$,  $(H_\alpha^l)^{p^i}$ represents $j_{E_\alpha}((H_\alpha^l)^{p^i})(\mc_\alpha(d_\alpha))$,
   which is a condition in $\Col(\overline{\kappa}_\alpha^+,<\kappa_\alpha)$ when $l=0$, and a condition in $\Col(\kappa_\alpha,\rho^+)$ for $l=1$.
   Recall that $\lambda=\rho^{++}$.
   These two collapse forcings both have size $\rho^+$ in $V$. The situation is slightly different for $(H_\alpha^2)^{p^i}$,
     a function whose domain is in $E_\alpha(\kappa_\alpha)$: if we let $i_\alpha: V \rightarrow N_\alpha = Ult(V, E_\alpha(\kappa_\alpha))$,
     then we can view this function as representing  $i_\alpha((H_\alpha^2)^{p^i})(\kappa_\alpha)$, a condition
     in $\Col(i_\alpha(s_\alpha)(\kappa_\alpha)^{+++}, < i_\alpha(\kappa_\alpha))^{N_\alpha}$. This poset has size $2^{\kappa_\alpha}$ in $V$.

 Recalling that $\eta < \kappa_0$, it follows that there are at most $\rho^+$ possibilities for the sequence
     of conditions
     $\langle j_{E_\alpha}((H_\alpha^0)^{p^i})(\mc_\alpha(d_\alpha)), j_{E_\alpha}((H_\alpha^1)^{p^i})(\mc_\alpha(d_\alpha)), i_\alpha((H_\alpha^2)^{p^i})(\kappa_\alpha):
     \alpha < \eta \rangle$.
     Hence, we can thin out the collection of the conditions so that for all $\alpha<\eta$, $i<\lambda$, and $l<3$,
     the functions $(H_\alpha^l)^{p^i}$ represent the same function in $M_\alpha$ or $N_\alpha$.
     It is now easy to show that the resulting collection of conditions is pairwise compatible.
\end{proof}

The following lemma is straightforward:

\begin{lemma}
$((\mathbb{P}/p) \setminus \alpha, \leq^*)$ is $\overline{\kappa}_\alpha^+$-closed.
\end{lemma}

The proof of the Prikry property here requires some ``integration" of the conditions.
The following lemmas (Lemma \ref{compute} and Lemma \ref{integrate}) show that we can assemble the conditions properly. We state the lemmas for a pure condition. In general, a condition may not be pure, and each impure coordinate divides the condition into blocks. The lemmas can be stated in each separate block individually.

\begin{lemma}\label{compute}
Let $p=\langle p_\alpha : \alpha<\eta \rangle$ be a pure condition and $\alpha \not \in \supp(p)$. Let $d=\dom(f_\alpha^p)$. Set $q=j_{E_\alpha}(p)+\mc_\alpha(d)$. Then $q \restriction \alpha= p \restriction \alpha$.
\end{lemma}

\begin{proof}

Fix $\beta<\alpha$. Since $|d_\beta^p|<\kappa_\alpha$, $j_{E_\alpha}(d_\beta^p)=j_{E_\alpha}[d_\beta^p]$. Hence, $d_\beta^q=\mc_\alpha(d)[j_{E_\alpha}(d_\beta^p)]=d_\beta^p$. Furthermore, for $\gamma \in d_\beta^q, f_\beta^q(\gamma)=j_{E_\alpha}(f_\beta^p) \circ (\mc_\beta(d))^{-1}(\gamma)=j_{E_\alpha}(f_\beta^p)(j_{E_\alpha}(\gamma))=j_{E_\alpha}(f_\beta^p(\gamma))=f_\beta^p(\gamma)$. This shows that $f_\beta^q=f_\beta^p$. A similar calculation shows that other components in $q_\beta$ are the same as the corresponding components in $p_\beta$.

\end{proof}

Further calculations show

\begin{lemma}\label{integrate}(The integration lemma)
Assume $\langle p_\alpha : \alpha<\eta \rangle$ is pure. Fix $\alpha<\eta$. $d= \dom(f_\alpha^p)$. Let $f \leq f_\alpha ^p$ with $\dom(f)=d^\prime$.  Let $A \in E_\alpha(d^\prime)$, and $A$ projects down to a subset of $A_\alpha^p$. For $\tau \in A$, denote $(p \restriction \alpha)_\tau=(p \restriction \alpha)_{\tau \restriction d}=(p +(\tau \restriction d ))\restriction \alpha$. Suppose for each $\tau \in A$, there is a condition $t(\tau) \leq^* (p \restriction \alpha)_{\tau}$, $h^0(\tau) \leq (H_\alpha^0)^p(\tau \restriction d)$, and $h^1(\tau) \leq (H_\alpha^1)^p (\tau \restriction d)$. Then there is a condition $q \leq^*p$ such that if $\psi \in A_\alpha^q$, $\tau=\psi \restriction d^\prime$, and $\nu=\tau(\kappa_\alpha)=\psi(\kappa_\alpha)$, we have:
\begin{enumerate}
        \item \label{in1} If $\beta<\alpha$ and $g$ is the Cohen part of $(q+\psi)_\beta$, then $g=f_\beta^{t(\tau)}$.
    \item \label{in2} $(q+\psi) \restriction \alpha \leq^*t(\tau)$.
    \item \label{in3}For $l=0,1$, $(H_\alpha^l)^q(\psi)=h^l(\tau)$.
    \item \label{in4}$(H_\alpha^2)^q(\nu)=(H_\alpha^2)^p(\nu)$ .
	\item \label{in5} $f_\alpha^q \leq f$.
\end{enumerate}
\end{lemma}

\begin{proof}
  Set $r = j_{E_\alpha}(t)(\mc_\alpha(d^\prime))$, where $t$ is considered as a function $\tau \mapsto t(\tau)$. Note that $r$ is a condition in the forcing up to $\alpha$ by a routine calculation. We recall that $r = \langle r_\beta: \beta<\alpha \rangle$, where for each $\beta<\alpha$,
\begin{center}
$r_\beta=\langle  f_\beta^r,A_\beta^r,(H_\beta^0)^r,(H_\beta^1)^r,(H_\beta^2)^r\rangle$.
\end{center}

 Fix $\beta<\alpha$. Here are some properties  of $r_\beta$:
\begin{enumerate}[label=(A\arabic*)]
\item \label{A1}
\begin{enumerate}
    \item \label{in1.1} Let $x_\beta=d_\beta^r \cap \kappa_\alpha$. Then $x_\beta$ is a bounded subset of $\kappa_\alpha$.
    \item \label{in1.2} Let $\gamma_\beta=\ot(d_\beta^r \setminus \kappa_\alpha)$. Then $\gamma_\beta<\kappa_\beta^+<\kappa_\alpha$.
    \item \label{in1.3} $A_\beta^r \in E_\beta(d_\beta^r)$.
\end{enumerate}
 \ref{in1.1} and \ref{in1.2} follow from the fact that $\dom(f_\beta^{t(\tau)})$ has size $\kappa_\beta$ for all $\tau$. 
 \ref{in1.3} follows from the fact that $j_{E_\alpha}(\mu \mapsto h_\alpha^\beta(\mu(\kappa_\alpha)))(\mc_\alpha(d^\prime))=E_\beta$.
 Fix an increasing enumeration of $d_\beta^r \setminus \kappa_\alpha$ as $\{\xi_{\beta,i}^r:i<\gamma_\beta\}$.  
For each $\tau \in A_\alpha^p$, let 
\begin{center}
    $t(\tau)=\langle\langle f_\beta^\tau,A_\beta^\tau,(H_\beta^0)^\tau,(H_\beta^1)^\tau,(H_\beta^2)^\tau \rangle : \beta<\alpha \rangle $
\end{center}
and let $d_\beta^\tau=\dom(f_\beta^\tau)$.

We record a few equations for each $\beta<\alpha$:

\begin{enumerate}[label=(\roman*)]
\item If $l=0,1$, we have 
  \label{in2.1}\begin{align*}
   j_{E_\beta}((H_\beta^l)^r)(\mc_\beta(d_\beta^r)) &= j_{E_\beta}(j_{E_\alpha}(\tau \mapsto (H_\beta^l)^\tau)(\mc_\alpha(d^\prime)))(j_{E_\alpha}(\tau \mapsto \mc_\beta(d_\beta^\tau)(\mc_\alpha(d^\prime)))) \\
&=j_{E_\alpha}(\tau \mapsto j_{h_\alpha^\beta(\tau(\kappa_\alpha))}(H_\beta^l)^\tau(\mc_\beta(d_\beta^\tau)))(\mc_\alpha(d^\prime)).
\end{align*}
\item \label{in2.2}\begin{align*}
    j_{E_\beta}((H_\beta^2)^r)(\kappa_\beta) &=j_{E_\beta}(j_{E_\alpha}(\tau \mapsto(H_\beta^2)^\tau)(\mc_\alpha(d^\prime)))(j_{E_\alpha}(\tau \mapsto \kappa_\beta)(\mc_\alpha(d^\prime))) \\
    &=j_{E_\alpha}(\tau \mapsto j_{h_\alpha^\beta(\tau(\kappa_\alpha))}(H_\beta^2)^\tau(\kappa_\beta))(\mc_\alpha(d^\prime)).
\end{align*}
\end{enumerate}

Note that by Lemma \ref{compute}, $A_\beta^r=j_{E_\alpha}(A_\beta^r)_{\mc_\alpha(d^\prime)}$ . We shrink $A$ to $A^*$ so that for every $\tau \in A^*$, and every $\beta<\alpha$, the following statements hold:
\item \label{A2}
\begin{enumerate}
    \item \label{in3.1}$d_\beta^\tau \cap \kappa_\alpha=x_\beta$.
    \item \label{in3.2}$\ot(d_\beta^\tau \setminus \kappa_\alpha)=\gamma_\beta$.
    \item \label{in3.3}For $\xi\in d_\beta^\tau \cap \kappa_\alpha$, $f_\beta^\tau(\xi)=f_\beta^r(\xi)$.
    \item  \label{in3.4} Let $\{\xi_{\beta,i}^\tau: i<\gamma_\beta\}$ be the increasing enumeration of $d_\beta^\tau \setminus \kappa_\alpha$, then for all $i<\gamma_\beta$, $f_\beta^\tau(\xi_{\beta,i}^\tau)=f_\beta^r(\xi_{\beta,i}^r)$.
    \item  \label{in3.5} $A_\beta^\tau=A_\beta^r \circ \tau^{-1}$.
    \item  \label{in3.6}For $l=0,1$,
 $j_{h_\alpha^\beta(\tau(\kappa_\alpha))}((H_\beta^l)^\tau)(\mc_\beta(d_\beta^\tau))=j_{E_\beta}((H_\beta^l)^r)(\mc_\beta(d_\beta^r))$.
    \item  \label{in3.7} $j_{h_\alpha^\beta(\tau(\kappa_\alpha))}((H_\beta^2)^\tau)(\kappa_\beta)=j_{E_\beta}((H_\beta^2)^r)(\kappa_\beta)$.
    \end{enumerate}

It is tempting to think we will set $q_\beta$ to be $r_\beta$ for all $\beta < \alpha$,
  but we need to make a tiny cosmetic modification. The property of $\tau$ in \ref{in3.6} and \ref{in3.7} makes sure that
  for each $\tau \in A^*$, $\beta<\alpha$, and $l=0,1$,
\begin{align*}
C_{\beta,l}:=\{\psi \in h_\alpha^\beta(\tau(\kappa_\alpha))(d_\beta^\tau): (H_\beta^l)^\tau(\psi)=(H_\beta^l)^r(\psi \circ \tau)\} \in h_\alpha^\beta(\tau(\kappa_\alpha))(d_\beta^\tau),
\end{align*}

and

\begin{align*}
C_{\beta,2}:=\{\psi \in h_\alpha^\beta(\tau(\kappa_\alpha))(d_\beta^\tau): (H_\beta^2)^\tau(\psi(\kappa_\beta))=(H_\beta^2)^r(\psi (\kappa_\beta)) \} \in h_\alpha^\beta(\tau(\kappa_\alpha))(d_\beta^\tau).    
\end{align*}

Shrink all measure one sets appearing in $t(\tau)$ so that every $\beta$-object belongs to $C_{\beta,0} \cap C_{\beta,1} \cap C_{\beta,2}$. Restrict the collapsing functions in the natural way, and call the result $t^*(\tau)$. Instead of integrating the function $t$ to obtain $r$, we will integrate the function $t^*$ in the same manner to obtain $\langle q_\beta : \beta<\alpha \rangle$.

Since we just shrank the measure one sets and restricted the collapses, all the properties
in the \ref{A2}-list hold, except that the property \ref{in3.5} is weakened to: $(A_\beta^\tau)^* \subseteq A_\beta^r \circ \tau^{-1}$, where $(A_\beta^\tau)^*$ is the measure one set of the $\beta$th-coordinate of $t^*(\mu)$

Now for $\beta<\alpha$, let $q_\beta=j_{E_\alpha}(t^*)(\mc_\alpha(d^\prime))$. Note that $q_\beta$ is almost identical to $r_\beta$, except the measure one sets: $A_\beta^q \subseteq A_\beta^r$. 
We now define $q_\alpha$. 
 Set $q_\alpha= \langle f_\alpha^q,A_\alpha^q,(H_\alpha^0)^q,(H_\alpha^1)^q,(H_\alpha^2)^q \rangle$ as follows:


\item \label{A3}
\begin{enumerate}
\item \label{in4.1} $f_\alpha^q$ has domain $d^\prime \cup  \bigcup\limits_{\beta<\alpha}(x_\beta \cup \{\xi_{\beta,i}^r: i<\gamma_\beta\})$. This is just $d^\prime \cup \bigcup_{\beta<\alpha}d_\beta^r$. 
\item \label{in4.2} For each $\xi \in \dom(f_\alpha^q)$, if $\xi \in d^\prime$, $f_\alpha^q(\xi)=f(\xi)$, otherwise $f_\alpha^q(\xi)=0$.

\item \label{in4.3} $A_\alpha^q \in E_\alpha(\dom(f_\alpha^q))$, where $\psi \in A_\alpha^q$ iff
\begin{enumerate}
    \item  \label{in4.3a}$\tau=\psi \restriction d^\prime \in A^*$,
    \item  \label{in4.3b}$\bigcup\limits_{\beta<\alpha}(x_\beta \cup \{\xi_{\beta,i}^r: i<\gamma_\beta\}) \subseteq \dom(\psi)$,
    \item  \label{in4.3c}for each $\beta<\alpha$ and  each $\xi \in x_\beta$, $\psi(\xi)=\xi$,
    \item  \label{in4.3d}for each $\beta$ and  $i$, $\psi(\xi_{\beta,i}^r)=\xi_{\beta,i}^\tau$, where
      $\tau$ is defined as in \ref{in4.3a}.
\end{enumerate}
\item \label{in4.4} For $l=0,1$, set $(H_\alpha^l)^q(\psi)=h^l(\tau)$.

\item \label{in4.5} $(H_\alpha^2)^q=(H_\alpha^2)^p \restriction A_\alpha^q(\kappa_\alpha)$.
\end{enumerate}
For $\beta>\alpha$, extend $p_\beta$ to $q_\beta= \langle f^q_\beta,A^q_\beta,(H^0)^q_\beta,(H^1)^q_\beta,(H^2)^q_\beta\rangle$ in the obvious way, meaning:
\item \label{A4}
\begin{enumerate}
    \item \label{in5.1} Extend the $f_\beta^p$ to $f_\beta^q$ with domain $d_\beta^q:=\dom(f_\alpha^q) \cup d^\prime$, $f_\beta^q(\gamma)=f_\beta^p(\gamma)$ if $\gamma \in d_\beta^p$, otherwise $f_\beta^q(\gamma)=0$.
    \item \label{in5.2}The measure one set in $q_\beta$ projects down to the measure one subsets appearing in $p_\beta$, i.e. $A^{q_\beta} \restriction d_\beta^p \subseteq A_\beta^p$. Also, $A^q_\beta \subseteq A_\beta(d_\beta^q)$.
    \item \label{in5.3} Restrict the collapses based on the measure one set we just defined, i.e. for $l=0,1$, $(H^l)_\beta^q(\tau)=(H^0)_\beta^p(\tau \restriction d_\beta^p)$, and $(H^2)_\beta^q=(H^2)_\beta^p \restriction A_\beta^q (\kappa_\beta)$.
\end{enumerate}

Let $q=\langle q_\beta : \beta< \eta\rangle$.
We claim $q$ satisfies the conclusion of Lemma \ref{integrate}.  It is easy to see that $q \setminus \alpha$ is $\leq^*$-below $p \setminus \alpha$. For $\tau \in A^*$, $t^*(\tau) \leq^* t(\tau) \leq^*(p+\tau \restriction d) \restriction \alpha $, and so $q\restriction \alpha \leq^* (j_{E_\alpha}(p)+\mc_\alpha(d))\restriction \alpha$. By Lemma \ref{compute},
the last term belongs to
$j_{E_\alpha}(\mathbb{P}_{\langle E_\beta: \beta<\alpha\rangle}) \restriction \lambda= \mathbb{P}_{\langle E_\alpha : \beta <\alpha \rangle}$,
and is equal to $p \restriction \alpha$. Hence, $q \leq^* p$.

Now we check that $q$ satisfies all the properties listed in Lemma \ref{integrate}.
Fix $\psi \in A_\alpha^q$, $\tau=\psi \restriction d^\prime$, $\mu=\psi \restriction d$, and $\nu=\mu(\kappa_\alpha)=\tau(\kappa_\alpha)=\psi(\kappa_\alpha)$. 

Requirement \ref{in1}: Fix $\beta<\alpha$. The Cohen part of $(q+\tau)_\beta$ is $f^q_\beta \circ \psi^{-1}=f^r_\beta \circ \psi^{-1}$. By  \ref{in4.3b} and \ref{in4.3c}, and \ref{in4.3d},
$\dom(f^r_\beta \circ \psi^{-1})=\psi[d^r_\beta]=\psi[x_\beta \cup \{\xi_{\beta,i}^r:i<\gamma_\beta\}]=x_\beta \cup \{\xi_{\beta,i}^\tau:i<\gamma_\beta\}$.
The last term is equal to $\dom(f_\beta^\tau)$ by \ref{in3.1} and \ref{in3.2}. From  \ref{in4.3d} and \ref{in3.3}, if $\xi \in x_\beta$, $f_\beta^r \circ \psi^{-1}(\xi)=f^r_\beta(\xi)=f_\beta^\tau(\xi)$. From \ref{in3.4}, for $i<\gamma_\beta$, $f_\beta^r \circ\ \psi^{-1}(\xi_{\beta,i}^\tau)=f_\beta^r(\xi_{\beta,i}^r)=f_\beta^\tau(\xi_{\beta,i}^\mu)$. The proof for requirement \ref{in1} is done.

Requirement \ref{in2}: Fix $\beta<\alpha$. From \ref{in3.5}, we have $A_\beta^q \circ \psi^{-1} \subseteq A_\beta^r \circ \psi^{-1}=A_\beta^r \circ \tau^{-1}=A_\beta^\tau$. By \ref{in4.3a}, $\tau \in A^*$. For $\sigma \in A_\beta^q \circ \psi^{-1}$, $\sigma \in (A_\beta^\tau)^*$, so $l=0,1$, $(H_\beta^l)^{q+\psi}(\sigma)=(H_\beta^l)^q(\sigma \circ \psi)=(H_\beta^l)^r(\sigma \circ \tau)=(H_\beta^l)^\tau(\sigma)$. Similarly, $(H_\beta^l)^{q+\psi}(\sigma(\kappa_\beta))=(H_\beta^2)^\tau(\sigma(\kappa_\beta))$.

Requirement \ref{in3}: From \ref{in4.4}, for $l=0,1$,
$(H_\alpha^l)^q(\psi)=h^l(\tau)$.

Requirement \ref{in4}: Straightforward from \ref{in4.5}.

Requirement \ref{in5}: Follows from \ref{in4.2}.

This completes the proof.

\end{enumerate}
\end{proof}

\section{prikry property}\label{prikry}

\begin{thm}
  $(\mathbb{P},\leq,\leq^*)$ has the Prikry property, that is to say for any boolean value $b$ and any
    condition $p \in \mathbb{P}$, there is a condition $p^\prime \leq^* p$ such that $p^\prime$ decides $b$.
\end{thm}

Our proof of the Prikry property follows the same lines as the proof in \cite{Gitikblow},
  but the collapse parts introduce additional challenges, which we briefly explain.  For each component $p_\alpha$,
  the collapse parts involve three collapsing forcings, where the chain condition of the first two collapse forcings
  is very close to the closure of the last collapse forcing. For $\alpha>0$, the collapse parts at $\alpha$  are only $\overline{\kappa}_\alpha^+$-closed,
  while the Cohen part is $\kappa_\alpha^+$-closed. This all makes it natural to group together components of a forcing condition
  which live on different levels, which explains why our inductive hypothesis concerns a product $\mathbb{P} \times \mathbb{A}$:
  the intuition is that the factor $\mathbb A$ anticipates some collapsing at the top level when we add another level to $\mathbb{P}$.

\begin{proof}
We consider 3 cases: 
\begin{enumerate}
    \item $\eta=1$.
    \item $\eta>1$ is a successor ordinal.
    \item $\eta$ is limit.
\end{enumerate}Throughout these 3 cases, we assume for simplicity that the condition $p$ is pure.
We prove a stronger statement by induction on $\eta$: Suppose $\mathbb{P}$ has length $\eta$, and $(\mathbb{A}, \leq )$ is a $\overline{\kappa}_\eta^+$-closed forcing poset. Define $(p,a) \leq^*(p^\prime,a^\prime)$ in $\mathbb{P} \times \mathbb{A}$ iff $p \leq^*p^\prime$ and $a \leq a^\prime$. Then $(\mathbb{P}\times \mathbb{A},\leq,\leq^*)$ also has
the Prikry property.

\smallskip

\noindent \underline{\textbf{CASE 1}: $\mathbb{\eta}=1$}

\smallskip  

We drop the subscript $0$ for simplicity, that is  $\kappa=\kappa_0$, $E=E_0$, $ \mc=\mc_0$, $s=s_0$ and so on.
Note here that $\mathbb{A}$ is $\kappa^+$-closed. Let $b$ be a boolean value. Let $p= \langle f,A,H^0,H^1,H^2 \rangle \in \mathbb{P}$, and $a\in \mathbb{A}$. Let $\theta$ be a sufficiently large regular cardinal, $N \prec H_\theta$ such that ${}^{< \kappa}N \subseteq N$, $|N|=\kappa$ and $p,a,\mathbb{P},\mathbb{A},b \in N$. Enumerate the dense open subsets of $(\mathcal{A}(\kappa,\lambda) \times \mathbb{A}) \cap N$ as $\{D_i : i<\kappa\}$. Build a decreasing sequence $\{(f^\prime_i,a_i) : i<\kappa\}$ in
$(\mathcal{A}(\kappa,\lambda) \times \mathbb{A}) \cap N$, where $f^\prime_0=f$, $a_0=a$  and $(f^\prime_{i+1},a_{i+1}) \in D_i$ for all $i$. Let $f^\prime=\bigcup\limits_{i<\kappa}f^\prime_i$, and $a^\prime$ be a lower bound of $\{a_i:i<\kappa\}$. Set $d^\prime=\dom(f^\prime)$. By a simple density argument, we have $d^\prime = N\cap \lambda$. We see that $[d^\prime]^{<\kappa}\subseteq N$. Let $A^\prime \in E(d^\prime)$ be such that $A^\prime \subseteq A(d^\prime)$ and $A^\prime$ projects down to a subset of $A$. Recall the property of $A(d^\prime): \{\mu \in A(d^\prime) : \mu(\kappa)=\nu\}$ has size at most $s(\nu)^{++}$.
For each $\mu \in A^\prime$, $\dom(\mu) \subseteq d^\prime \subseteq N$ and $|\dom(\mu)| \le \mu(\kappa)<\kappa$, hence, $\mu \in N$.

Let $\nu \in A^\prime(\kappa)$. By Lemma \ref{sparseuf}, let $\{ \mu_j,h^0_j,h^1_j\}_{j < \lambda(\nu)}$ be an enumeration in
$N$ of the triples $\mu,h^0,h^1$ where $\mu \in A^\prime$, $\mu(\kappa)=\nu$, $h^0 \in \Col(\omega_1,<\nu)$, and $h^1 \in \Col(\nu,\rho(\nu)^+)$, respectively.

Define $D_\nu$ as the collection of $(g,x) \in \mathcal{A}(\kappa,\lambda) \times \mathbb{A}$ such that there is an $h \in \Col(\lambda(\nu)^+,<\kappa)$ with $h \leq H^2(\nu)$ meeting the following requirements:

\begin{enumerate}
    \item For all $j<\lambda(\nu)$, $\dom(\mu_j) \subset \dom(g)$.
    \item For all $j<\lambda(\nu)$,
    \begin{itemize}
        \item EITHER $(\langle  g\oplus \mu_j, \lambda(\nu),h^0_j,h^1_j,h \rangle,x)$ decides $b$,
        \item OR there are no $g^\prime \leq g$, $h^\prime \leq h,$ and $x^\prime \leq x$ such that 
        
         $(\langle g^\prime \oplus \mu_j,\lambda(\nu),h^0_j,h^1_j,h^\prime \rangle,x^\prime)$ decides $b$.
    \end{itemize}
\end{enumerate}

\begin{claim}
$D_\nu$ is a dense open subset of $\mathcal{A}(\kappa,\lambda)\times \mathbb{A}$ and $D_\nu \in N$.
\end{claim}

\begin{claimproof}
  Since $D_\nu$ is defined using parameters in $N$, $D_\nu \in N$. It is easy to check that
  $D_\nu$ is open. Now we check the density for $D_\nu$. Let $g \in \mathcal{A}(\kappa,\lambda)$ and $x \in \mathbb{A}$. Because $\mathcal{A}(\kappa,\lambda)$ is $\kappa^+$-closed, we may assume $g$ meets the first requirement. Build sequences $\vec{g}=\{g_j\}_{j\leq \lambda(\nu)}$, $\vec{h}=\{h_j\}_{j \leq \lambda(\nu)}$, and $\vec{x}=\{x_j\}_{j \leq \lambda(\nu)}$ such that

\begin{itemize}
    \item $\vec{g}$ is a decreasing sequence in $\mathcal{A}(\kappa,\lambda)$.
    \item $\vec{h}$ is a decreasing sequence in $\Col(\lambda(\nu)^+,<\kappa)$.
\item $\vec{x}$ is a decreasing sequence in $\mathbb{A}$.
       \item $g_0=g$, $h_0=H^2(\nu)$ and $x_0=x$.
        \item At each limit $j \leq \lambda(\nu)$, we take $g_j=\bigcup\limits_{j^\prime<j}g_{j^\prime}$, and $h_j= \bigcup\limits_{j^\prime<j}h_{j^\prime}$.
\item At each limit $j\leq \lambda(\nu)$, take $x_j$ as a lower bound of $\{x_{j^\prime}:j^\prime <j\}$.
\end{itemize}

Note that the construction proceeds to the end since
$\mathcal{A}(\kappa,\lambda)$ and $\mathbb{A}$ are $\kappa^+$-closed, 
  $\Col(\lambda(\nu)^+,<\kappa)$ is $\lambda(\nu)^+$-closed, and $\lambda(\nu)<\kappa$.
Now suppose $f_j$, $h_j$, and $x_j$ are constructed and $j<\lambda(\nu)$.
Ask if there is a triple $g^\prime$,$h^\prime$,$x^\prime$ below $g_j$,$h_j$, and $x_j$, respectively, such that ($\langle g^\prime \oplus \mu_j,\lambda(\nu),h^0_j,h^1_j,h^\prime \rangle,x^\prime)$ decides $b$. If the answer is no, take $g_{j+1}=g_j$,$h_{j+1}=h_j$, and $x_{j+1}=x_j$.  Otherwise, there are such $g^\prime$,$h^\prime$, and $x^\prime$. Take $g_{j+1}=g^\prime$, $h_{j+1}=h^\prime$, and $x_{j+1}=x^\prime$. From the construction, we see $(g_{\lambda(\nu)},x_{\lambda(\nu)})\leq (g,x)$ is in $D_\nu$, as witnessed by $h_{\lambda(\nu)}$.
\end{claimproof}

By the construction of $f^\prime$ and $a^\prime$ we have $(f^\prime,a^\prime) \in D_\nu$ with a witness $h \in \Col(\rho(\nu)^{+3}, <\kappa)$.
Define $(H^2)^\prime(\nu)=h$.

\smallskip
We record the properties of $f^\prime$, $a^\prime$, and $(H^2)^\prime$ here:

\noindent $(\star)$ For each $\mu \in A^\prime$ with $\nu=\mu(\kappa)$, we have that for all $h^0 \in \Col(\omega_1,<\nu)$ and $h^1 \in \Col(\nu,\rho(\nu)^+)$,

\begin{itemize} 
\item EITHER $(\langle f^\prime \oplus \mu,\lambda(\nu),h^0,h^1,(H^2)^\prime(\nu)\rangle,a^\prime)$ decides $b$,
\item OR there are no $g\leq f^\prime$, $h\leq (H^2)^\prime(\nu)$, and $x\leq a^\prime$ such that
$(\langle g\oplus \mu,\lambda(\nu),h^0,h^1,h \rangle,x) $ decides $b$.
\end{itemize}

\medskip

  Now for each $\mu \in A^\prime$, find $g(\mu)\leq f^\prime$, $h^0(\mu) \leq H^0(\mu \restriction \dom(f))$, $h^1(\mu) \leq H^1(\mu \restriction \dom(f))$,
  $h^2(\mu)  \leq (H^2)^\prime(\mu(\kappa))$, and $x(\mu) \leq a^\prime$ such that
  
  \begin{center}
  $(\langle g(\mu) \oplus \mu, \lambda(\mu), h^0(\mu), h^1(\mu), h^2(\mu) \rangle, x(\mu))$ decides $b$.
\end{center}

By property $(\star)$, we have

\begin{center}
    $(\langle f^\prime \oplus \mu, \lambda(\nu), h^0(\mu), h^1(\mu), (H^2)^\prime(\mu(\kappa))\rangle,a^\prime)$ decides $b$. 
\end{center}
Set 
\begin{align*}
    B_0&=\{\mu \in A^\prime : (\langle f^\prime \oplus \mu,\lambda(\mu),h^0(\mu),h^1(\mu),(H^2)^\prime(\mu(\kappa))\rangle ,a^\prime)\Vdash b \}. \\
    B_1 &=\{\mu \in A^\prime : (\langle f^\prime \oplus \mu,\lambda(\mu),h^0(\mu),h^1(\mu),(H^2)^\prime(\mu(\kappa))\rangle ,a^\prime)\Vdash \neg b \}.
\end{align*}
We see that $A^\prime=B_0 \sqcup B_1$. Choose $i_0 \in \{0,1\}$ such that $B_{i_0} \in E(d^\prime)$. Define $p^{\prime\prime}= \langle f^{\prime\prime},A^{\prime\prime},(H^0)^{\prime\prime},(H^1)^{\prime\prime},(H^2)^{\prime\prime}\rangle$ (note that $f^\prime, A^\prime$, and $(H^2)^\prime$ are already defined) as follows:

\begin{itemize}
\item $f^{\prime\prime}=f^\prime$.
\item $A^{\prime\prime}=B_{i_0}$.
\item For $l=0,1$, $\dom((H^l)^{\prime\prime})=A^{\prime\prime}$ and $H^l(\mu)=h^l(\mu)$.
\item $(H^2)^{\prime\prime}=(H^2)^\prime \restriction A^{\prime\prime} (\kappa)$. 
\end{itemize}

\begin{claim}
$(p^{\prime\prime},a^\prime) \leq^*(p,a)$ and
$(p^{\prime\prime},a^\prime)$ decides $b$. 
\end{claim}

\begin{claimproof}
  It is easy to see that $(p^{\prime\prime},a^\prime) \leq^*(p,a)$. Let $q \leq p^{\prime\prime}, x\leq a^\prime $ such that $(q,x)$ decides $b$.
    Without loss of generality, $(q,x) \Vdash b$, and $q$ is not pure. Hence, $q \leq^* p^{\prime\prime}+\mu$ for some $\mu \in A^{\prime\prime}$.  Observe that $f^q \leq g \oplus \mu$ for some $g \leq f^\prime$, and $(h^2)^q \leq (H^2)^{\prime\prime}(\mu(\kappa))=(H^2)^\prime(\mu(\kappa))$.  From the properties $(\star)$, we have 
\begin{center}$(\langle f^\prime \oplus \mu,\lambda(\mu),(h^0)^q,(h^1)^q,(H^2)^\prime(\mu(\kappa))\rangle,a^\prime) \Vdash b$.
\end{center}
Since for $l=0,1$, $(h^l)^q \leq (H^l)^{\prime\prime}(\mu)=h^l(\mu)$, $\mu$ can't be in $B_1$.
Hence, $i_0=0$, and $A^{\prime\prime} = B_0$.

By a similar argument, every extension of $(p^{\prime\prime},a^\prime)$ by $\mu^\prime\in A^{\prime\prime}$  forces $b$. Since every extension of $p^{\prime\prime}$ has a further extension which is not pure, and that forces $b$, by a density argument, $(p^{\prime\prime},a^\prime) \Vdash b$.
\end{claimproof}

\smallskip

\noindent \underline{\textbf{CASE 2}: $\mathbb{\eta}>1$ is a successor ordinal} 

\smallskip

The proofs for all successor ordinals $\eta>1$ are essentially the same.
For simplicity, assume $\eta=2$. Hence, $\mathbb{A}$ is $\kappa_1^+$-closed.
Suppose for simplicity that $p =  \langle p_0,p_1 \rangle$ is pure.
Write $p_0=\langle f_0,A_0,H_0^0,H_0^1,H_0^2 \rangle$, and $p_1= \langle f_1,A_1,H_1^0,H_1^1,H_2^1 \rangle$.  Also, let $a\in \mathbb{A}$.

Let $\theta$ be a sufficiently large regular cardinal. Let $N_1 \prec H_\theta$, $|N_1|=\kappa_1, {}^{<\kappa_1}N_1 \subseteq N_1,$ and $p,a,\mathbb{P},b \in N_1$. Enumerate the dense open subsets of $\mathcal{A}(\kappa_1,\lambda) \times \mathbb{A} \cap N_1$ as $\{D_i : i<\kappa\}$. Build a decreasing sequence $\{(f_{1,i}^\prime,a_i) : i<\kappa_1\}$
in $(\mathcal{A}(\kappa_1,\lambda)\times \mathbb{A}) \cap N_1$
such that $f_{1,0}^\prime=f_1$, $a_0=a$, and $(f_{1,i+1}^\prime,a_{i+1}) \in D_i$ for all $i$. Let $f_1^\prime =\bigcup\limits_{i<\kappa_1} f_{1,i}^\prime$
and $a^\prime=$ a lower bound of $\{a_i:i<\kappa_1\}$.  Set $d_1^\prime=\dom(f_1^\prime)$, which is $N_1 \cap \lambda$. Let $A_1^\prime \in E_1(d_1^\prime)$ be such that $A_1^\prime \subseteq A_1(d_1^\prime)$ and $A_1^\prime$ projects down to a subset of $A_1$. Similar to the one extender case, $A_1^\prime \subseteq N_1$.

Let $ \nu \in A_1^\prime(\kappa_1)$. By Lemma \ref{sparseuf}, let $\{t_j,h^0_j,h^1_j,\mu_j\}_{j<\lambda_1(\nu)}$ be an enumeration in $N_1$ of the quadruples $t,\mu,h^0,h^1$ where
$t \in \mathbb{P}_{h_1^0(\nu)}$, $\mu \in A^\prime$ with $\mu(\kappa_1)=\nu$, $h^0 \in \Col(\kappa_0^+,<\nu)$, and $h^1 \in \Col(\nu, \rho_1(\nu)^+)$, respectively.

Define $D_\nu$ as the collection of $(g,x) \in \mathcal{A}(\kappa_1,\lambda)\times \mathbb{A}$ such that there is an $h\in \Col(\lambda_1(\nu)^+,<\kappa_1)$ with $h \leq H_1^2(\nu)$ meeting the following requirements:

\begin{enumerate}
    \item For all $j<\lambda_1(\nu)$, $\dom(\mu_j) \subset \dom(g)$.
    \item For all $j<\lambda_1(\nu)$,
    \begin{itemize}
        \item EITHER  $(t_j^\frown \langle g \oplus \mu_j,\lambda_1(\nu),h^0_j,h^1_j,h \rangle,x)$ decides $b$,
        \item OR there is no $g^\prime \leq g$, $h^\prime \leq h$ and $x^\prime\leq x$ such that $(t_j^\frown \langle g^\prime \oplus \mu_j,\lambda_1(\nu),h^0_j,h^1_j,h \rangle,x^\prime)$ decides $b$.
    \end{itemize}
\end{enumerate}

    Similar to the one extender case, $D_\nu$ is a dense open subset of $\mathcal{A}(\kappa_1,\lambda) \times \mathbb{A}$ and $D_\nu \in N_1$. We have $(f_1^\prime,a^\prime) \in D_\nu$ with a witness $h \in \Col(\rho_1(\nu)^{+3},<\kappa_1)$. Define $(H_1^2)^\prime(\nu)=h$.

\smallskip
    We record some properties of $(H_1^2)^\prime$:

\noindent $(\star)$ For each $\mu \in A_1^\prime$ with $\nu=\mu(\kappa_1)$, we have that for all $h^0 \in \Col(\kappa_0^+,<\nu)$, $h^1 \in \Col(\nu,\rho_1(\nu)^+)$, and $t\in \mathbb{P}_{h_1^0(\nu)}$:

\begin{itemize} 
\item EITHER $(t^\frown \langle f^\prime_1 \oplus \mu,\lambda_1(\nu),h^0,h^1,(H^2_1)^\prime(\nu)\rangle,a^\prime)$ decides $b$,
\item OR there is no
  $g\leq f^\prime_1$, $h\leq (H^2_1)^\prime(\nu)$, and $x\leq a^\prime$ such that
$(t^\frown\langle g\oplus \mu,\lambda_1(\nu),h^0,h^1,h \rangle,x) $ decides $b$.
\end{itemize}

\smallskip

Fix $\mu \in A_1^\prime$. Let $\dot{G}$ be a canonical name for a generic object for $\mathbb{P}_{h_1^0(\mu(\kappa_1))} \times \Col(\kappa_0^+,<\mu(\kappa_1)) \times \Col(\mu(\kappa_1),\rho_1(\mu)^+)$.  Since the product of the collapses $\Col(\kappa_0^+,<\mu(\kappa_1)) \times \Col(\mu(\kappa_1),\rho(\mu)^+)$ is $\kappa_0^+$-closed, we apply the induction hypothesis for the one extender case to the condition 
        \begin{center}
        $((p_0)_\mu,(H_1^0)(\mu\restriction \dom(f_1)),(H_1^1)(\mu \restriction \dom(f_1)))$.
        \end{center}
        Let $t(\mu)\leq^*(p_0)_\mu$, $h^0(\mu)\leq (H_1^0)(\mu\restriction \dom(f_1))$, and $h^1(\mu)\leq (H_1^1)(\mu\restriction \dom(f_1))$  such that
\begin{center}
$(t(\mu),h^0(\mu),h^1(\mu))$ decides if $\exists (t,(h^*){}^0,(h^*){}^1) \in \dot{G}$, $(t ^\frown \langle f_1^\prime,\lambda_1(\mu),(h^*){}^0,(h^*){}^1,(H^2_1)^\prime(\mu(\kappa_1)) \rangle,a^\prime) \parallel b$.
\end{center}    

By strengthening $t(\mu)$ (under $\leq^*)$, $h^0(\mu)$ and $h^1(\mu)$ further, we can assume that for each $\mu \in A_1^\prime$, $(t(\mu),h^0(\mu),h^1(\mu))$ satisfies exactly one out of the following three mutually exclusive properties:

\begin{enumerate}[label=(P\arabic*)]
    \item \label{P1} $(t(\mu),h^0(\mu),h^1(\mu)) \Vdash 
    \exists (t,(h^0)^*),(h^1)^*) \in \dot{G}$, \\ $(t ^\frown \langle f_1^\prime \oplus \mu ,\lambda_1(\mu),(h^0)^*,(h^1)^*,(H^2_1)^\prime(\mu(\kappa_1)) \rangle,a^\prime) \Vdash b$.
    
     \item \label{P2} $(t(\mu),h^0(\mu),h^1(\mu))  \Vdash\exists (t,(h^0)^*),(h^1)^*) \in \dot{G}$, \\ $(t ^\frown \langle f_1^\prime \oplus \mu ,\lambda_1(\mu),(h^0)^*,(h^1)^*,(H^2_1)^\prime(\mu(\kappa_1)) \rangle,a^\prime) \Vdash \neg b$.
     
     \item  \label{P3} $(t(\mu),h^0(\mu),h^1(\mu)) \Vdash \nexists (t,(h^0)^*),(h^1)^*) \in \dot{G}$, \\ $(t ^\frown \langle f_1^\prime \oplus \mu ,\lambda_1(\mu),(h^0)^*,(h^1)^*,(H^2_1)^\prime(\mu(\kappa_1)) \rangle,a^\prime) \parallel b$.
\end{enumerate}

Shrink $A_1^\prime$ so that every $\mu \in A_1^\prime$ falls into the same case above. Use Lemma \ref{integrate} to find $q \leq^* p_0 ^\frown \langle f_1^\prime,A_1^\prime,H_1^0,H_1^1,(H_1^2)^\prime \rangle$ such that for each $\tau \in A_1^q$ with $\mu=\tau \restriction d_1^\prime$, we have $f_0^{q+\tau}=f_0^{t(\mu)}$, $(q+\tau) \restriction 1 \leq^* t(\mu)$, $(H_1^0)^q(\tau)=h^0(\mu)$, $(H_1^1)^q(\tau)=h^1(\mu)$, $(H_1^2)^q(\mu(\kappa_1))=(H_1^2)^\prime(\mu(\kappa_1))$, and $f_1^q \leq f_1^\prime$.

\begin{claim}
$(q,a^\prime)$ decides $b$  
\end{claim}

\begin{claimproof}
  Let $(r,x) \leq (q,a^\prime)$ be such that $(r,x) \parallel b$. Without loss of generality $(r,x) \Vdash b$. Assume also $1 \in \supp(r)$. Then $r$ is an extension of a 1-step extension of $q$ by $\mu^\prime \in A_1^q$ for some $\mu^\prime$.

  Recall $d_1^\prime=\dom(f_1^\prime)$. Set $\mu=\mu^\prime \restriction d_1^\prime$. We see that $f_1^r  \restriction d_1^\prime \leq f_1^\prime \oplus \mu,$ $(h_1^2)^r \leq (H_1^2)^\prime(\mu(\kappa_1))$, and $x \leq a^\prime$.
By $(\star) $, we have $(r_0 ^\frown \langle f_1^\prime \oplus \mu, \lambda_1(\mu),(h_1^0)^r,(h_1^1)^r,(H_1^2)^\prime(\mu(\kappa_1))\rangle ,a^\prime) \Vdash b$.  

    Since $r$ is an extension of a one-step extension of $q$ by $\mu^\prime$,
    $r_0 \leq t(\mu)$, $(h_1^0)^r \leq h^0(\mu)$, and $(h_1^1)^r \leq h^1(\mu)$.
    By construction $\mu$ must fall into exactly one of the cases \ref{P1}, \ref{P2}, \ref{P3} listed above:
    we claim that cases \ref{P2} and \ref{P3} are impossible. Suppose for contradiction that we are in case \ref{P2},
    and force below $(r_0, (h_1^0)^r, (h_1^1)^r)$ to obtain a generic object $G$. Then $(r_0, (h_1^0)^r, (h_1^1)^r) \in G$
    and $r_0^\frown \langle f_1^\prime \oplus \mu, \lambda_1(\mu),(h_1^0)^r,(h_1^1)^r,(H_1^2)^\prime(\mu(\kappa_1))\rangle ,a^\prime) \Vdash b$,  
    but $(r_0, (h_1^0)^r, (h_1^1)^r) \leq (t(\mu), h^0(\mu), h^1(\mu))$ and we are in case \ref{P2} so that also
    $G$ contains a triple $(t, (h^0)^*, (h^1)^*)$
    with
    $(t^\frown \langle f_1^\prime \oplus \mu ,\lambda_1(\mu),(h^0)^*,(h^1)^*,(H^2_1)^\prime(\mu(\kappa_1)) \rangle,a^\prime) \Vdash \neg b$:
    this gives comparable conditions forcing $b$ and $\neg b$, an immediate contradiction. Similarly, we are not in case \ref{P3},
    so $(t(\mu),h^0(\mu),h^1(\mu))$ must have the property \ref{P1}.

  We claim already $r^\mu:=(t(\mu) ^\frown \langle f_1^\prime \oplus \mu, \lambda_1(\mu),h^0(\mu),h^1(\mu),(H_1^2)^\prime(\mu(\kappa_1)) \rangle,a^\prime)$ forces $b$. Suppose not, find $(r^\prime,x^\prime) \leq r^\mu$ such that $(r^\prime,x^\prime)  \Vdash \neg b$. Strengthen further, we assume $f_1^{r^\prime} \restriction d_1^\prime \leq f_1^\prime \oplus \mu$, $(h_1^2)^{r^\prime}\leq (H_1^2)^\prime(\mu(\kappa_1))$, and $x^\prime \leq a^\prime$. By $(\star)$, we have
\begin{equation} \label{notb}
      (r^\prime_0{}^\frown \langle f^\prime_1 \oplus,\mu,\lambda_1(\mu),(h^0_1)^{r^\prime},(h^1_1)^{r^\prime}(H_1^2)^\prime(\mu(\kappa_1)) \rangle ,a^\prime) \Vdash \neg b. \tag*{$(\dagger)$}
 \end{equation}
  
Let $G$ be a generic object for 
$\mathbb{P}_{h_1^0(\mu(\kappa_1))}\times \Col(\kappa_0^+,<\mu(\kappa_1)) \times \Col(\mu(\kappa_1),\rho_1(\mu)^+)$ containing
$(r_0^\prime,(h_1^0)^{r^\prime},(h_1^1)^{r^\prime})$.  Since $(r^\prime_0,(h_1^0)^{r^\prime},(h_1^1)^{r^\prime}) \leq (t(\mu),h^0(\mu),h^1(\mu))$, we have $(t(\mu),h^0(\mu),h^1(\mu)) \in G$. Then find $(t,(h^0)^*,(h^1)^*) \in G$ according to the property \ref{P1},  i.e. 

\begin{equation} \label{b}
(t ^\frown \langle f_1^\prime \oplus \mu,\lambda_1(\mu),(h^0)^*,(h^1)^*,(H^2)^\prime(\mu(\kappa_1)) \rangle ,a^\prime)\Vdash b. \tag*{$(\dagger\dagger)$}
\end{equation}
By the directedness of $G$, we may assume $t \leq r_0^\prime, (h^0)^* \leq (h_1^0)^{r^\prime}$, and $(h^1)^* \leq (h_1^1)^{r^\prime}$.
Thus, from \ref{notb}, we have

  \[
  (t^\frown \langle f_1^\prime \oplus \mu,\lambda_1(\mu),(h^0)^*,(h^1)^*,(H_1^2)^\prime(\mu(\kappa_1))\rangle,a^\prime) \Vdash \neg b, 
\]
which is a contradiction when comparing to \ref{b}.

With the way we shrank the measure one set $A_1^\prime$, we have that for every $\widetilde{\mu}\in A_1^q$,$(t(\widetilde{\mu}\restriction d_1^\prime),h^0(\widetilde{\mu}\restriction d_1^\prime),h^1(\widetilde{\mu}\restriction d_1^\prime))$
has the property (P1). By the same proof, $(q+\widetilde{\mu} ,a^\prime)\Vdash b$. Finally, we are going to show that $(q,a^\prime) \Vdash b$. If not, let $(q^\prime,x^\prime) \leq (q,a^\prime)$ such that $(q^\prime,x^\prime) \Vdash \neg b$, but we can extend further so that $1 \in \supp(q^\prime)$. Hence, $(q^\prime,x^\prime) \leq (q+\widetilde{\mu},a^\prime)$ for some $\widetilde{\mu}\in A_1^q$, but $(q+\widetilde{\mu},a^\prime) \Vdash b$, a contradiction.

Hence, we have finished the proof for $\eta$ a successor ordinal.

\end{claimproof}

\smallskip

\noindent \underline{\textbf{CASE 3}: $\mathbb{\eta}$ is limit}
\smallskip

The proofs for the limit cases are essentially the same, we may assume $\eta=\omega$. Suppose for simplicity that $p= \langle p_n : n<\omega \rangle$ is pure. Write 
\begin{center}
$p_n=\langle f_n^p,A_n^p,(H_n^0)^p,(H_n^1)^p,(H_n^2)^p \rangle$ for each $n$.
\end{center}
Recall $\mathbb{A}$ is $\overline{\kappa}_\omega^+$-closed. Let $a \in \mathbb{A}$. We will build inductively a $\leq^*$-sequence $\{ (q^m,a^m): m<\omega \} $ where $q^0 \leq^* p$, $a^0 \leq a$. We will show that $(q,a^\prime)$, the $\leq^*$-lower bound of the sequence $\{(q^m,a^m):m<\omega\}$, will decide $b$.
We divide our constructions into two parts: $q^0,a^0$, and $q^m,a^m$ for positive $m$.

\textbf{Construction of} $q^0$: Define
\begin{center}
$\mathbb{Q}_0=\{(f,\vec{r}) \in (\mathcal{A}(\kappa_0,\lambda)/f_0^p,\leq)\times ((\mathbb{P}/p)\setminus 1,\leq^*): \dom(f) \mbox{ is a subset of $\dom(f_1^r)$ }\}$.
\end{center}

Here note that $f_1^r$ is the first Cohen part of $\vec{r}$. Observe that $\mathbb{Q}_0$ is $\kappa_0^+$-closed. Let $\theta$ be a sufficiently large regular cardinal. Let $N_0 \prec H_\theta$ be such that $|N_0|=\kappa_0$, ${}^{<\kappa_0}N_0 \subseteq N_0$, $p,a,\mathbb{Q}_0,\mathbb{P},\mathbb{A} \in N_0$. Enumerate dense open sets in $(\mathbb{Q}_0 \times \mathbb{A})\cap N_0$ as $\{D_i: i<\kappa_0\}$. Build a $\mathbb{Q}_0 \times \mathbb{A}$-decreasing sequence $\{(f^\prime_{0,i},\vec{r}_{0,i},a_i):i<\kappa_0\}$, each in $N_0$, where $f^\prime_{0,0}=f_0^p$, $\vec{r}_{0,0}=p \setminus 1$, $a_0=a$, and for all $i$, $(f^\prime_{0,i+1},\vec{r}_{0,i+1},a_{i+1}) \in D_i$. Let $(f_0^\prime,\vec{r}_0,a^0)$ be a lower bound of the sequence $\{(f^\prime_{0,i},\vec{r}_{i},a_i):i<\kappa_0\}$. Set $d_0^\prime=\dom(f_0^\prime)$, which is just $N_0 \cap \lambda$. Let $A_0^\prime\in E_0(d_0^\prime)$ be such that $A_0^\prime \subseteq A_0(d_0^\prime)$ and $A_0^\prime$ projects down to a subset of $A_0^p$. As usual, $A_0^\prime \subseteq N_0$.

Let $ \nu \in A_0^\prime(\kappa_0)$. By Lemma \ref{sparseuf}, let $\{\mu_j,h_j^0,h_j^1\}_{j<\lambda_0(\nu)}$ be an enumeration in $N_0$ of the triples $\mu,h^0,h^1$ where $\mu\in A_0^\prime$, $\mu(\kappa_0)=\nu$, $h^0 \in \Col(\omega_1,<\nu)$, and $h^1\in \Col(\nu,\rho_0(\nu)^+),$ respectively.

Define $D_\nu$ as the collection of $(g,\vec{r},x)\in \mathbb{Q}_0 \times \mathbb{A}$ such that there is an $h \in \Col(\lambda_0(\nu)^+,<\kappa_0)$ with $h \leq(H_0^2)^p(\nu)$ meeting the following requirements:

\begin{enumerate}
    \item For all $j<\lambda_0(\nu)$, $\dom(\mu_j) \subset \dom(g)$.
    \item For all $j<\lambda_0(\nu)$, 
    \begin{itemize}
        \item EITHER $(\langle g \oplus \mu_j,\lambda_0(\nu),h_j^0,h_j^1,h \rangle ^\frown \vec{r},x) $ decides $b$,
        \item OR there is no $g^\prime \leq g$, $h^\prime \leq h$,  $\vec{r}{}^\prime \leq^*\vec{r}$, and $x^\prime \leq x$ such that $\dom(g^\prime)$ is a subset of the Cohen part of $\vec{r}{}^\prime$, and the condition 
        $(\langle g^\prime \oplus \mu_j,\lambda_0(\nu),h_j^0,h_j^1,h^\prime \rangle ^\frown \vec{r}{}^\prime,x^\prime) $ decides $b$. 
    \end{itemize}
\end{enumerate}
Similar to the one-extender case, the set $D_\nu$ is a dense, open subset of $\mathbb{Q}_0 \times \mathbb{A}$, and $D_\nu \in N_0$. Hence, we have $(f_0^\prime,\vec{r}_0,a^0) \in D_\nu$ with a witness $h \in \Col(\lambda_0(\nu)^+,<\kappa_0)$. Define $(H_0^2)^\prime(\nu)=h$.

We record some properties of $(H_0^2)^\prime$: 

\medskip

\noindent $(\star_0)$: 
For each $\mu \in A_0^\prime$ with $\nu=\mu(\kappa_0)$, we have that for all $h^0 \in \Col(\omega_1,<\nu)$ and $h^1 \in \Col(\nu,\rho_0(\nu)^+)$, 
\begin{itemize}
    \item EITHER $(\langle f_0^\prime \oplus \mu,\lambda_0(\nu),h^0,h^1,(H^2_0)^\prime(\nu) \rangle ^\frown \vec{r}_0,a^0 \rangle $  decides $b$,
    \item OR there is no $g \leq f_0^\prime,h\leq (H_0^2)^\prime(\nu)$,  $\vec{r}\leq^* \vec{r}_0$ such that $\dom(g)$ is a subset of the Cohen part of $\vec{r}$, and $x\leq a^0$ such that $(\langle g \oplus \mu,\lambda_0(\nu),h^0,h^1,h \rangle ^\frown \vec{r},x)$ decides $b$. 
\end{itemize}

Now for each $\mu \in A_0^\prime$, find  $h^0(\mu) \leq H^0(\mu \restriction d_0^p)$, $h^1(\mu) \leq H^1(\mu \restriction d_0^p)$,  (if they exist), such that 
\begin{center}
    $(\langle f^\prime_0 \oplus \mu,\lambda_0(\mu),h^0(\mu),h^1(\mu),(H_0^2)^\prime(\mu(\kappa_0)) \rangle^\frown \vec{r}_0,a^0)$ decides b,
\end{center}
otherwise, $h^0(\mu)=H^0_0(\mu \restriction d_0^p)$, $h^1(\mu)=H^1_0(\mu \restriction d_0^p)$.

\medskip

Now set
\begin{align*}
    B_0 & = \{\mu \in A_0^\prime : (\langle f_0^\prime \oplus \mu,\lambda_0(\mu),h^0(\mu),h^1(\mu),(H^2_0)^\prime(\mu(\kappa_0))\rangle ^\frown \vec{r}_0,a^0) \Vdash b\} \\
    B_1 & = \{\mu \in A_0^\prime : (\langle f_0^\prime \oplus \mu,\lambda_0(\mu),h^0(\mu),h^1(\mu),(H^2_0)^\prime(\mu(\kappa_0))\rangle ^\frown \vec{r}_0,a^0)\Vdash  \neg b\} \\
    B_2 & = \{\mu \in A_0^\prime : (\langle f_0^\prime \oplus \mu,\lambda_0(\mu),h^0(\mu),h^1(\mu),(H^2_0)^\prime(\mu(\kappa_0))\rangle ^\frown \vec{r}_0,a^0) \nparallel b\}
\end{align*}

We see that $A_0^\prime=B_0 \sqcup B_1 \sqcup B_2$. As in  the case of one extender, choose $i_0\in\{0,1,2\}$ such that $B_{i_0}\in E_0(d_0^\prime)$. 
Define $f_0^{\prime\prime}=f_0^\prime$. Let $A_0^{\prime\prime}=B_{i_0}$. For $l=0,1$, set $\dom(H_0^l)^{\prime\prime}=A^{\prime\prime}_0$, and $(H_0^l)^{\prime \prime}(\mu)=h^l(\mu)$. Let $(H^2)^{\prime\prime}=(H^2)^\prime \restriction A_0^{\prime\prime}(\kappa_0)$.

Finally, extend $\vec{r}_0$ to $\vec{r}_0{}^{\prime}$ in a natural way: make sure $\dom(f^{\prime\prime})$ is a subset of the Cohen part in $\vec{r}_0^{\prime}$, $f_m^{\vec{r}_0^{\prime}}(\xi)=f_m^{\vec{r}_0}(\xi)$ if $\xi$ is in $\dom(f_m^{\vec{r}_0})$, otherwise $f_m^{\vec{r}_0^\prime}(\xi)=0$. All measure one sets project down to subsets of the corresponding measure one sets in $\vec{r}_0$, meaning $A_m^{\vec{r}_0^\prime} \restriction \dom(f_m^{\vec{r}_0}) \subseteq A_m^{\vec{r}_0}$. If $\mu$ belongs to $A^{\vec{r}_0^{\prime}}_m$, define $(H^0_m)^{\vec{r}_0^\prime}(\mu)=(H^0_m)^{\vec{r}_0}(\mu \restriction \dom(f_m^{\vec{r}_0}))$ $(H^1_m)^{\vec{r}_0^\prime}(\mu)=(H^1_m)^{\vec{r}_0}(\mu \restriction \dom(f_m^{\vec{r}_0}))$, and $H^2_m(\nu)=(H^2_m)^{\vec{r}_0^\prime}(\nu)=(H^2_m)^{\vec{r}_0}(\nu)$, see more details at \ref{A4} of Lemma \ref{integrate}. Set $q^0=\langle f^{\prime\prime},A_0^{\prime\prime},(H_0^0)^{\prime\prime},(H_0^1)^{\prime\prime},(H_0^2)^{\prime\prime} \rangle ^\frown \vec{r}_0^\prime$. This finishes the construction of $q^0$ and $a^0$.

\textbf{Construction of} $q^{m+1}$: Suppose $q^m$ and $a^m$ are constructed. Define 

\begin{center}
    $\mathbb{Q}_{m+1}=\{(f,\vec{r}) \in (\mathcal{A}(\kappa_{m+1},\lambda)/f_{m+1}^{q^m}),\leq) \times (\mathbb{P}/q^m \setminus (m+2),\leq^*): \dom(f) \mbox{ is a subset of } \dom(f_{m+2}^r)\}$.
\end{center}

Here note that $f_{m+2}^r$ is the first Cohen part of $\vec{r}$. $\mathbb{Q}_{m+1}$ is $\kappa_{m+1}^+$-closed. Let $\theta$ be a sufficiently large regular cardinal. Let $N_{m+1}\prec H_\theta$ be such that $|N_{m+1}|=\kappa_{m+1}$, ${}^{<\kappa_{m+1}}N_{m+1}\subseteq N_{m+1}$, $q^m,a^m,\mathbb{P},\mathbb{A}\in N_{m+1}$ (so $\mathbb{Q}_{m+1} \in N_{m+1}$).

Enumerate the dense open sets in $(\mathbb{Q}_{m+1} \times \mathbb{A}) \cap N_{m+1}$ as $\{D_i:i<\kappa_{m+1}\}$. Build a $\mathbb{Q}_{m+1} \times \mathbb{A}$-decreasing sequence $\{(f^\prime_{m+1,i},\vec{r}_{m+1,i},a^m_i): i<\kappa_{m+1}\}$, each in $N_{m+1}$, where $f^\prime_{m+1,0}=f_{m+1}^{q^m}$, $\vec{r}_{m+1,0}=q^m \setminus (m+2)$, $a^m_0=a^m$, and for all $i$, $(f^\prime_{m+1,i+1},\vec{r}_{m+1,i+1},a^m_{i+1}) \in D_i$. Let $(f^\prime_{m+1},\vec{r}_{m+1},a^{m+1})$ be the greatest lower bound of the sequence $\{(f^\prime_{m+1,i},\vec{r}_{m+1,i},a^m_{i}): i<\kappa_{m+1}\}$.  $d_{m+1}^\prime=\dom(f_{m+1}^\prime)$, which is $N_{m+1}\cap \lambda$. Let $A_{m+1}^\prime \in E_{m+1}(d_{m+1}^\prime)$ be such that $A_{m+1}^\prime \subseteq A_{m+1}(d_{m+1}^\prime)$, and $A_{m+1}^\prime$ projects down to a subset of $A_{m+1}^{q^m}$. As usual, $A_{m+1}^\prime \subseteq N_{m+1}$.

Let $\nu \in A_{m+1}^\prime(\kappa_{m+1})$. By Lemma \ref{sparseuf}, let $\{t_j,\mu_j,h^0_j,h^1_j\}_{j<\lambda_{m+1}(\nu)}$ be an enumeration in $N_{m+1}$ of the quadruples $t,\mu,h^0,h^1$ where $t \in \mathbb{P}_{\langle h_{m+1}^n(\nu):n\leq m \rangle}$, $\mu\in A_{m+1}^\prime$ with $\mu(\kappa_{m+1})=\nu$, $h^0 \in \Col(\kappa_m^+,<\nu)$, and $h^1 \in \Col(\nu,\rho_{m+1}(\nu)^+)$, respectively.

Define $D_\nu$ as the collection of $(g,\vec{r},x)\in \mathbb{Q}_{m+1} \times \mathbb{A}$ such that there is an $h \in \Col(\lambda_{m+1}(\nu)^+,<\kappa_{m+1})$ with $h\leq (H_{m+1}^2)^{q^m}(\nu)$ meeting the following requirements:

\begin{enumerate}
    \item For all $j<\lambda_{m+1}(\nu)$, $\dom(\mu_j) \subset \dom(g)$.
    \item For all $j<\lambda_{m+1}(\nu)$,
    \begin{itemize}
        \item EITHER $(t^\frown\langle g \oplus \mu_j,\lambda_{m+1}(\nu),h^0_j,h^1_j,h \rangle ^\frown \vec{r},x)$ decides $b$,
        \item OR there is no $g^\prime \leq g$, $h^\prime \leq h$, $\vec{r}{}^\prime \leq^* \vec{r}$, and $x^\prime \leq x$ such that $\dom(g^\prime)$ is a subset of the Cohen part of $\vec{r}{}^\prime$, and the condition \\ $(t^\frown \langle g^\prime \oplus \mu_j,\lambda_{m+1}(\nu),h^0_j,h^1_j,h^\prime \rangle^\frown \vec{r}{}^\prime,x^\prime)$ decides $b$.
    \end{itemize}
\end{enumerate}
Similar to the other cases, the set $D_\nu$ is a dense, open subset of $\mathbb{Q}_{m+1} \times \mathbb{A}$, and
$D_\nu \in N_{m+1}$. Hence, we have $(f_{m+1}^\prime,\vec{r}_{m+1},a^{m+1}) \in D_\nu$ with a witness $h \in \Col(\lambda_{m+1}(\nu)^+,<\kappa_{m+1})$. Define $(H^2_{m+1})^\prime(\nu)=h$.

We record some properties of $(H^2_{m+1})^\prime$: 

\smallskip

$(\star_{m+1})$ 
 For each $\mu \in A_{m+1}^\prime$ with $\nu=\mu(\kappa_{m+1})$, we have that for all $t \in \mathbb{P}_{\langle h_{m+1}^n(\nu): n \leq m \rangle}$, $h^0\in \Col(\kappa_m^+,<\nu)$, and $h^1 \in \Col(\nu,\rho_{m+1}(\nu)^+)$,
\begin{itemize}
    \item EITHER $(t^\frown \langle f_{m+1}^\prime \oplus \mu,\lambda_{m+1}(\nu),h^0,h^1,(H^2_{m+1})^\prime(\mu)\rangle ^\frown \vec{r}_{m+1},a^{m+1})$ decides $b$,
    \item OR there is no $g\leq f_{m+1}^\prime, h\leq (H^2_{m+1})^\prime(\nu)$, $\vec{r} \leq^* \vec{r}_{m+1}$, and $x\leq a^{m+1}$ such that $(t^\frown \langle g \oplus \mu,\lambda_{m+1}(\nu),h^0,h^1,h \rangle^\frown \vec{r},x)$ decides $b$.
\end{itemize}

\smallskip

Fix $\mu \in A_1^\prime$. Let $\dot{G}$ be a canonical name for a generic object for $\mathbb{P}_{\langle h_{m+1}^n(\mu(\kappa_{m+1})):n\leq m \rangle} \times \Col(\kappa_m^+,<\mu(\kappa_{m+1})) \times \Col(\mu(\kappa_{m+1}),\rho_{m+1}(\mu)^+)$. Let $(\vec{q}^m_{\leq m})_\mu=\langle ((q^m)_n)_\mu:n \leq m)$, which is just the first $m+1$ coordinates of $q^m$ squished by $\mu$.  Since the product of collapses $\Col(\kappa_m^+,<\mu(\kappa_{m+1})) \times \Col(\mu(\kappa_{m+1}),\rho_{m+1}(\mu)^+)$ is $\kappa_m^+$-closed, we apply the induction hypothesis for the Prikry property for the $m$-extender case to the condition 

\begin{center}
$((\vec{q}^m_{\leq m})_\mu,(H_{m+1}^0)^{q^m}(\mu \restriction \dom(f_{m+1}^{q^m})),(H_{m+1}^1)^{q^m}(\mu \restriction \dom(f_{m+1}^{q^m})))$.
\end{center}

Let $t(\mu) \leq^* (\vec{q}^m_{\leq m})_\mu$, $h^0(\mu)\leq (H_{m+1}^0)^{q^m}(\mu \restriction \dom(f_{m+1}^{q^m}))$, and \\ $h^1(\mu) \leq (H_{m+1}^1)^{q^m}(\mu \restriction \dom(f_{m+1}^{q^m}))$ be such that 

\begin{center}
    $(t(\mu),h^0(\mu),h^1(\mu))$ decides if $\exists (t,(h^0)^*,(h^1)^*) \in \dot{G}$, \\
    $(t^\frown \langle f_{m+1}^\prime,\lambda_{m+1}(\mu),(h^0)^*,(h^1)^*,(H^2_{m+1})^\prime (\mu(\kappa_{m+1})) \rangle ^\frown \vec{r}_{m+1},a^{m+1}) \parallel b$.
\end{center}

By strengthening $t(\mu)$ (under $\leq^*$), $h^0(\mu)$, and $h^1(\mu)$ further, we can assume that for each $\mu \in A_{m+1}^\prime$, $(t(\mu),h^0(\mu),h^1(\mu))$ satisfies exactly one out of the following three mutually exclusive properties: 

\begin{enumerate}[label=(Q\arabic*)]
    \item \label{Q1} $(t(\mu),h^0(\mu),h^1(\mu)) \Vdash \exists (t,(h^0)^*,(h^1)^*) \in \dot{G}$, \\
    $(t^\frown \langle f_{m+1}^\prime \oplus \mu,\lambda_{m+1}(\mu),(h^0)^*,(h^1)^*,(H^2_{m+1})^\prime(\mu(\kappa_{m+1}))\rangle ^\frown \vec{r}_{m+1},a^{m+1}) \Vdash b$.
    \item \label{Q2} $(t(\mu),h^0(\mu),h^1(\mu)) \Vdash \exists (t,(h^0)^*,(h^1)^*) \in \dot{G}$, \\
    $(t^\frown \langle f_{m+1}^\prime \oplus \mu,\lambda_{m+1}(\mu),(h^0)^*,(h^1)^*,(H^2_{m+1})^\prime(\mu(\kappa_{m+1}))\rangle ^\frown \vec{r}_{m+1},a^{m+1}) \Vdash \neg b$.
    \item \label{Q3} $(t(\mu),h^0(\mu),h^1(\mu)) \Vdash \nexists (t,(h^0)^*,(h^1)^*) \in \dot{G}$, \\
    $(t^\frown \langle f_{m+1}^\prime \oplus \mu,\lambda_{m+1}(\mu),(h^0)^*,(h^1)^*,(H^2_{m+1})^\prime(\mu(\kappa_{m+1}))\rangle ^\frown \vec{r}_{m+1},a^{m+1}) \parallel b$.
\end{enumerate}

Shrink $A_{m+1}^\prime$ so that every $\mu\in A_{m+1}^\prime$ falls into the same case above.

Use Lemma \ref{integrate} to find a condition $q^{m+1}$ such that 
\begin{center}
    $q^{m+1} \leq ^* \langle q^m_n:n \leq m \rangle^\frown \langle f_{m+1}^\prime,A_{m+1}^\prime,(H_{m+1}^0)^{q^m}(H_{m+1}^1)^{q^m},(H_{m+1}^2)^\prime \rangle^\frown \vec{r}_{m+1}$
\end{center}
satisfying all properties stated in Lemma \ref{integrate}, where $t(\mu),h^0(\mu),h^1(\mu)$ are as described right before introducing the properties \ref{Q1},\ref{Q2}, and \ref{Q3}, which means that if $\tau \in A_{m+1}^{q^{m+1}}$ and $\mu=\tau \restriction d_{m+1}^\prime$, and $\nu=\mu(\kappa_{m+1})$, then for $k \leq m$, $f_k^{q^{m+1}+\tau}=f_k^{t(\mu)}$, $(q^{m+1} + \tau) \restriction (m+1) \leq^* t(\mu)$, for $l=0,1$, $(H_{m+1}^l)^{q^{m+1}}(\tau)=h^0(\mu)$, $(H_{m+1}^2)^{q^{m+1}}(\nu)=(H_{m+1}^2)^\prime(\nu)$, and $f_{m+1}^{q^{m+1}} \leq f_{m+1}^\prime$. We have finished the construction of $q^{m+1}$, and $a^{m+1}$.

Recall we take $(q,a^\prime)$ as a lower bound of the sequence $\{(q^m,a^m):m<\omega\}$. 

\begin{claim}
 There is a direct extension of $(q,a^\prime)$ which decides $b$.
 \end{claim}

 If we can find a direct extension of $(q,a^\prime)$ deciding $b$, then the proof is done. Suppose this is not the case. Let $(r,x) \leq (q,a^\prime)$ be such that $(r,x) \Vdash b$ and $r$ is not pure.  Suppose $r$ is not pure, which has the least possible value of $\max(\supp(r))$. Our proof is divided into 2 cases: $\max(\supp(r))=0$, and $\max(\supp(r))>0$.

\underline{CASE A} $\max(\supp(r))=0$. 

This means $\supp(r)=1$. Hence, $r \leq^* q+\mu^\prime$ for some $\mu^\prime \in A_0^q$.

We refer to all the notations in the construction of $q^0$. Set $\mu=\mu^\prime \restriction d_0^\prime$, and $\nu=\mu(\kappa_0)$. We have $r \setminus 1 \leq^* q \setminus 1 \leq^*\vec{r}{}^\prime_0\leq^*\vec{r}_0$. Similarly, we can trace back to see that $f^r \leq f_0^\prime \oplus \mu$. Also, $(h_0^2)^r \leq (H_2^0)^\prime(\nu)$ and $x \leq a^\prime \leq a^0$. By the property $(\star_0)$, $(\langle f_0^\prime \oplus \mu,\lambda_0(\nu),(h^0)^r,(h^1)^r,(H_0^2)^\prime(\nu)\rangle^\frown \vec{r}_0,a^0 ) \Vdash b$. This means $B_{i_0}=B_0$. Thus, for every $\widetilde{\mu} \in A_0^q$, we have

\begin{equation}\label{genb}
    (f_0^\prime \oplus (\widetilde{\mu}\restriction d_0^\prime),\lambda_0(\widetilde{\mu}),h^0_0(\widetilde{\mu}),h^1_0(\widetilde{\mu},(H^2_0)^\prime(\widetilde{\mu}(\kappa_0))^\frown \vec{r}_0,a^0) \Vdash b. \tag*{$(\dagger)$}
\end{equation}

One can check that $(q+\widetilde{\mu},a^\prime)$ is stronger than the condition in \ref{genb}. Therefore $(q+\widetilde{\mu},a^\prime) \Vdash b$. Since every extension of $(q, a^\prime)$ is compatible with $(q+\widetilde{\mu},a^\prime)$ for some $\widetilde{\mu}$,
  $(q, a^\prime) \Vdash b$. 

\underline{CASE B} $\max(\supp(r))=m+1$ for some $m<\omega$. 

We will refer to all the notations 
in the construction of $q^{m+1}$. Suppose part of the extension used some $\mu^\prime \in A_{m+1}^q$. Set $\mu=\mu^\prime \restriction d_{m+1}^\prime$ and $\nu=\mu(\kappa_{m+1})$. By tracing back, and lemma 5, the following properties hold for $r$ and $x$:
\begin{enumerate}
    \item \label{Prik1} $r \restriction (m+1) \leq t(\mu)$.
    \item \label{Prik2}$f_{m+1}^r \leq f_{m+1}^\prime \oplus \mu$.
    \item \label{Prik3}For $l=0,1$, $(h_{m+1}^l)^r \leq h^l(\mu)$.
    \item \label{Prik4}$(h_{m+1}^2)^r \leq (H_{m+1}^2)^\prime(\nu)$.
    \item \label{Prik5}$r \setminus (m+2) \leq^* \vec{r}_{m+1}$.
    \item \label{Prik6}$x \leq a^{m+1}$.
\end{enumerate}

By (\ref{Prik2}),(\ref{Prik4}),(\ref{Prik5}), and (\ref{Prik6}), and by $(\star_{m+1})$, we have

\begin{center}
 $(r \restriction (m+1)^\frown \langle f_{m+1}^\prime \oplus \mu, \lambda_{m+1}(\mu),(h_{m+1}^0)^r,(h_{m+1}^1)^r,(H_{m+1}^2)^\prime(\nu) \rangle ^\frown \vec{r}_{m+1},a^{m+1}) \Vdash b$.
\end{center}

By (\ref{Prik1}) and (\ref{Prik3}), $(t(\mu),h^0(\mu),h^1(\mu))$ has the property \ref{Q1}. Set 
\begin{center}
    $r^\mu:=t(\mu)^\frown\langle f_{m+1}^\prime \oplus \mu,\lambda_{m+1}(\nu),h^0(\mu),h^1(\mu),(H_{m+1}^2)^\prime(\mu(\kappa_{m+1}))\rangle ^\frown \vec{r}_{m+1}$ 
\end{center}

We claim that already $(r^\mu,a^{m+1}) \Vdash b$. Suppose not. Find $r^\prime \leq r,x^\prime \leq a^{m+1}$ such that $(r^\prime,x^\prime) \Vdash \neg b$. let $G$ be a generic extension of $\mathbb{P}_{\langle h_{m+1}^n(\nu):n \leq m \rangle}$ containing $( r^\prime \restriction (m+1),(h_{m+1}^0)^{r^\prime},(h_{m+1}^1)^{r^\prime})$, hence, containing $(t(\mu),h^0(\mu),h^1(\mu))$. By
the property \ref{Q1}, we can find $(t,(h^0)^*,(h^1)^*)\in G$ below $( r^\prime \restriction (m+1),(h_{m+1}^0)^{r^\prime},(h_{m+1}^1)^{r^\prime})$ such that 

\begin{equation}\label{gennotb}
(t^\frown \langle f_{m+1}^\prime \oplus \mu,\lambda_{m+1}(\mu),(h^0)^*,(h^1)^*,(H^2_{m+1})^\prime(\mu(\kappa_{m+1}))\rangle ^\frown \vec{r}_{m+1},a^{m+1}) \Vdash b. \tag*{$(\dagger\dagger)$}
\end{equation}

Since 
\begin{enumerate}
    \item $t \leq r^\prime \restriction (m+1)$.
    \item $f_{m+1}^{r^\prime} \leq  f_{m+1}^\prime \oplus \mu$.
    \item for $l=0,1,$ $(h^*)^l \leq (h_{m+1}^l)^{r^\prime}$,
    \item $(h_{m+1}^2)^{r^\prime} \leq (H_{m+1}^2)^\prime(\mu(\kappa_{m+1}))$.
    \item $r^\prime \setminus (m+2) \leq \vec{r}_{m+1}$, and
    \item $x^\prime \leq a^{m+1}$.
\end{enumerate}

Combining the fact that $(r^\prime,x^\prime) \Vdash \neg b$, and \ref{gennotb}, we have

\begin{equation*}
    (t^\frown \langle f_{m+1}^{r^\prime},\lambda_{m+1}(\nu),(h^0)^*,(h^1)^*,(h_{m+1}^2)^{r^\prime} \rangle ^\frown r^\prime \setminus (m+2),x^\prime) \Vdash b,\neg b,
\end{equation*}
which is a contradiction.

To show that $(q,a^\prime)$ forces $b$, note that every extension of $(q,a^\prime)$ is compatible with $(q+\widetilde{\mu},a^\prime)$ for some $\widetilde{\mu}^\prime \in A_{m+1}^{q^m}$. Hence, it is enough to show that every extension of $(q,a^\prime)$ by $\widetilde{\mu}^\prime \in A_{m+1}^{q^m}$ forces $b$.

Let $\widetilde{\mu}^\prime \in A_{m+1}^{q^m}$. Let $\widetilde{\mu}=\widetilde{\mu}^\prime \restriction d_{m+1}^\prime$ and $\widetilde{\nu}=\widetilde{\mu}(\kappa_{m+1})$. By the way we shrank $A_{m+1}^\prime$, $(t(\widetilde{\mu}),h^0(\widetilde{\mu}),h^1(\widetilde{\mu}))$ has the property \ref{Q1}. Similar proof as above shows

\begin{center}
($t(\widetilde{\mu})^\frown\langle f_{m+1}^\prime \oplus \widetilde{\mu},\lambda_{m+1}(\widetilde{\nu}),h^0(\widetilde{\mu}),h^1(\widetilde{\mu}),(H_{m+1}^2)^\prime(\widetilde{\nu})\rangle ^\frown \vec{r}_{m+1},a^{m+1}) \Vdash b$. 
\end{center}

Hence, already $(q+\widetilde{\mu}^\prime,a^\prime) \Vdash b$. This completes the proof of the Prikry property.

\end{proof}

A similar proof shows the following statement, which is known as the
  ``strong Prikry property''.

\begin{lemma}\label{strongprikry}
       Let $p\in\mathbb{P}$, and $D$ be a dense open subset of $\mathbb{P}$. Then there is $p^\prime \leq^*p$, and a finite set $I \subseteq \eta$ (can be empty) such that $I\cap \supp(p^\prime)=\emptyset$, and for each $\vec{\mu}\in \prod_{\alpha \in I}A_\alpha^{p^\prime}$, each $\mu_i$ addable, $p^{\prime}+\vec{\mu}\in D$.
\end{lemma}

\begin{proof}
  (Sketch) As in the proof of the Prikry property, we induct on the a stronger statement: by induction on $\eta$, if $\mathbb{A}$ is $\overline{\kappa}_\eta^+$- closed, then for each $(p,a)\in \mathbb{P} \times \mathbb{A}$, and a dense open set $D \subseteq \mathbb{P} \times \mathbb{A}$, there is a condition $(p^\prime,a^\prime)  \leq^* (p,a)$ and a finite set $I \subseteq \eta$ (can be empty) such that $I \cap \supp(p)=\emptyset$, and for each $\vec{\mu} \in \prod_{\alpha \in I}A_\alpha^{p^\prime}$, each $\mu_i$ is addable, and $(p^\prime + \vec{\mu},a^\prime) \in D$. We assume for simplicity that $p$ is pure. The elements of the proof for the case $\eta=\omega$ contain all the elements from the other cases. We will show only the case $\eta=\omega$. The proof has the same style as the proof of the
Prikry property, we assume $\mathbb{A}$ is trivial, and remove $\mathbb{A}$ from the proof to make the proof more readable.

  We only emphasise the key different ingredients from the proof of the Prikry property. For more details, look at the proof of the Prikry property.

Assume $p=\langle \langle f_n^p,A_n^p,(H_n^0)^p,(H_n^1)^p,(H_n^2)^p \rangle \rangle$ is a pure condition.
We will build a $\leq^*$-decreasing sequence $\langle q_m : m<\omega \rangle$,
  it will then be routine to check that a lower bound of the sequence $\langle q_m : m<\omega \rangle$ will satisfy the condition for the strong Prikry property.

\textbf{Construction of} $q^0$: Let 
\begin{center}
  $\mathbb{Q}_0=\{(f,\vec{r}) \in (\mathcal{A}(\kappa_0,\lambda)/f_0^p,\leq)\times ((\mathbb{P}/p)\setminus 1,\leq^*): \dom(f) \mbox{ is a subset of $\dom(f_1^r)$ }\}$.
\end{center}

Note that $f_1^r$ is the first Cohen part of $\vec{r}$. Fix a sufficiently large regular cardinal $\theta$. Build an elementary submodel $N_0 \prec H_\theta$ of size $\kappa_0$ closed under $<\kappa_0$-sequences containing enough information. Let $(f_0^\prime,\vec{r}_0)$ be $(N_0,\mathbb{Q}_0)$-generic. Let $d_0^\prime=\dom(f_0^\prime)$ and $A_0^\prime\in E_0(d_0^\prime)$ be such that $A_0^\prime \subseteq A_0(d_0^\prime)$ and $A_0^\prime$ projects down to a subset of $A_0^p$.

Let $\nu \in A^\prime_0(\kappa_0)$. As usual, let $\{\mu_j,h_j^0,h_j^1\}_{j<\lambda_0(\nu)}$ be an enumeration in $N_0$ of the triples $\mu,h^0,h^1$ where $\mu\in A_0^\prime$, $\mu(\kappa_0)=\nu$, $h^0 \in \Col(\omega_1,<\nu)$, and $h^1\in \Col(\nu,\rho_0(\nu)^+),$ respectively.

Define $D_\nu$ as the collection of $(g,\vec{r})\in \mathbb{Q}_0$ such that there is an $h \in \Col(\lambda_0(\nu)^+,<\kappa_0)$ with $h \leq(H_0^2)^p(\nu)$ meeting the following requirements:

\begin{enumerate}
    \item For all $j<\lambda_0(\nu)$, $\dom(\mu_j) \subset \dom(g)$.
    \item For all $j<\lambda_0(\nu)$, 
    \begin{itemize}
        \item EITHER $(\langle g \oplus \mu_j,\lambda_0(\nu),h_j^0,h_j^1,h \rangle ^\frown \vec{r}) \in D$,
        \item OR there is no $g^\prime \leq g$, $h^\prime \leq h$, and $\vec{r}{}^\prime \leq^*\vec{r}$ such that $\dom(g^\prime)$ is a subset of the Cohen part of $\vec{r}{}^\prime$, and the condition 
       
        $(\langle g^\prime \oplus \mu_j,\lambda_0(\nu),h_j^0,h_j^1,h^\prime \rangle ^\frown \vec{r}^\prime) \in D$. 
    \end{itemize}
\end{enumerate}
Then $D_\nu$ is a dense open subset of $\mathbb{Q}_0$ and is in $N_0$. Hence, $(f_0^\prime,\vec{r}_0) \in D_\nu$, with a witness $h\in \Col(\lambda_0(\nu)^+,<\kappa_0)$. Define $(H_0^2)^\prime(\nu)=h$.

Now fix $\mu \in A_0^\prime$. Find $h^0(\mu) \leq (H^0_0)^p(\mu \restriction d_0^p)$ and $h^1(\mu) \leq H^1(\mu \restriction d_0^p)$ (if exist) such that there are $f\leq f_0^\prime$, $h \leq (H_0^2)^\prime(\mu(\kappa_0))$, and $\vec{r} \leq^* \vec{r}_0$,

\begin{center}
$\langle f^\prime \oplus \mu,\lambda_0(\mu),h^0(\mu),h^1(\mu),h)\rangle ^\frown \vec{r} \in D$.
\end{center}

Hence,

\begin{center}
$\langle f_0^\prime \oplus \mu,\lambda_0(\mu),h^0(\mu),h^1(\mu),(H_0^2)^\prime(\mu(\kappa_0))\rangle ^\frown \vec{r}_0 \in D$.
\end{center}
  
Otherwise, set $h^0(\mu)=(H_0^0)^p(\mu \restriction d_0^p)$, $h^1(\mu)=(H_0^2)^p(\mu \restriction d_0^p)$. Define $(H_0^0)^\prime(\mu)=h^0(\mu)$ and $(H_0^1)^\prime(\mu)=h^1(\mu)$. Shrink $A_0^\prime$ to $A_0^{\prime\prime}$ so that

\begin{enumerate}[label=(R\arabic*)]

\item \label{R1} EITHER for every $\mu \in A_0^{\prime\prime}$, there are $f^\prime \leq f_0^\prime$, $h \leq (H_0^2)^\prime(\mu(\kappa_0))$, and $\vec{r} \leq^* \vec{r}_0$ such that $\langle f^\prime \oplus \mu, \lambda_0(\mu),h^0(\mu),h^1(\mu),h) \rangle ^\frown \vec{r} \in D$,
\item \label{R2} OR for every $\mu \in A_0^{\prime\prime}$, for every $f^\prime \leq f_0^\prime$, $h \leq (H_0^2)^\prime(\mu(\kappa_0))$, and $\vec{r} \leq^* \vec{r}_0$ such that $\langle f^\prime \oplus \mu, \lambda_0(\mu),h^0(\mu),h^1(\mu),h) \rangle ^\frown \vec{r} \not \in D$.
\end{enumerate}

Finally, define $q_0= \langle f_0^\prime, A_0^{\prime\prime},(H_0^0)^\prime \restriction A_0^{\prime\prime},(H_0^1) \restriction A_0^{\prime\prime},(H_0^2)^\prime \restriction A_0^{\prime\prime}(\kappa_0) \rangle ^\frown \vec{r}_0$. Here is the property of $q^0$: if $q^\prime$ is an extension of $q^0$ with $\supp(q^\prime)=\{0\}$ and $q^\prime \in D$, then for every $\tau \in A_0^{q^0}$, $q^0+ \tau \in D$. 

\textbf{Construction of} $q^{m+1}$: Suppose $q^m$ is constructed. Define 

\begin{center}

    $\mathbb{Q}_{m+1}=\{(f,\vec{r}) \in (\mathcal{A}(\kappa_{m+1},\lambda)/f_{m+1}^{q^m}),\leq) \times (\mathbb{P}/q^m \setminus (m+2),\leq^*): \dom(f) \mbox{ is a subset of } \dom(f_{m+2}^r)\}$.
\end{center}

Here note that $f_{m+2}^r$ is the first Cohen part of $\vec{r}$. Fix a sufficiently large regular cardinal $\theta$. Build an elementary submodel $N_{m+1} \prec H_\theta$ of size $\kappa_{m+1}$ closed under $<\kappa_{m+1}$-sequences and containing enough information. Let $(f_{m+1}^\prime,\vec{r}_{m+1})$ be $(N_{m+1},\mathbb{Q}_{m+1})$-generic. Let $d_{m+1}^\prime=\dom(f_{m+1}^\prime)$ and $A_{m+1}^\prime \in E_{m+1}(d_{m+1}^\prime)$ be such that $A_{m+1}^\prime \subseteq A_{m+1}(d_{m+1}^\prime)$ and $A_{m+1}^\prime$ projects down to a subset of $A_{m+1}^{q^m}$. 

Let $\nu \in A_{m+1}^\prime(\kappa_{m+1})$. As usual, let $\{t_j,\mu_j,h^0_j,h^1_j\}_{j<\lambda_{m+1}(\nu)}$ be an enumeration in $N_{m+1}$ of the quadruples $t,\mu,h^0,h^1$ where $t \in \mathbb{P}_{\langle h_{m+1}^n(\nu):n\leq m \rangle}$, $\mu\in A_{m+1}^\prime$ with $\mu(\kappa_{m+1})=\nu$, $h^0 \in \Col(\kappa_m^+,<\nu)$, and $h^1 \in \Col(\nu,\rho_{m+1}(\nu)^+)$, respectively.

Define $D_\nu$ as the collection of $(g,\vec{r}) \in \mathbb{Q}_{m+1} $ such that there is an $h \in \Col(\lambda_{m+1}(\nu)^+,<\kappa_{m+1})$ with $h\leq (H_{m+1}^2)^{q^m}(\nu)$ meeting the following requirements:

\begin{enumerate}
    \item For all $j<\lambda_{m+1}(\nu)$, $\dom(\mu_j) \subset \dom(g)$.
    \item For all $j<\lambda_{m+1}(\nu)$,
    \begin{itemize}
        \item EITHER $t_j^\frown\langle g \oplus \mu_j,\lambda_{m+1}(\nu),h^0_j,h^1_j,h \rangle ^\frown \vec{r} \in D$,
        \item OR there is no $g^\prime \leq g$, $h^\prime \leq h$, and $\vec{r}{}^\prime \leq^* \vec{r}$, such that $\dom(g^\prime)$ is a subset of the Cohen part of $\vec{r}{}^\prime$, and the condition \\ $t_j^\frown \langle g^\prime \oplus \mu_j,\lambda_{m+1}(\nu),h^0_j,h^1_j,h^\prime \rangle^\frown \vec{r} \in D$.
    \end{itemize}
\end{enumerate}
The set $D_\nu$ is a dense, open subset of $\mathbb{Q}_{m+1}$, and $D_\nu \in N_{m+1}.$ Hence, we have $(f_{m+1}^\prime,\vec{r}_{m+1})\in D_\nu$ with a witness $h \in \Col(\lambda_{m+1}(\nu)^+,<\kappa_{m+1})$. Define $(H^2_{m+1})^\prime(\nu)=h$.

Now outside $N_{m+1}$, let $E_\mu$ be the collection of
$(t,h^0,h^1) \in \mathbb{P}_{\langle h_{m+1}^n(\nu) : n\leq m \rangle} \times \Col(\kappa_m^+,<\nu) \times \Col(\nu,\rho_{m+1}(\nu)^+)$ such that EITHER

\begin{center}
$t ^\frown \langle f_{m+1}^\prime \oplus \mu, \lambda_{m+1}(\nu),h^0,h^1,(H_{m+1}^2)^\prime(\nu)\rangle ^\frown \vec{r}_{m+1} \in D$,
\end{center}

OR for all $g \leq f_{m+1}^\prime \oplus \mu$, $h^2 \leq (H_{m+1}^2)^\prime(\mu)$, and $\vec{r}^{\prime} \leq^* \vec{r}_{m+1}$,
\begin{center}
$t ^\frown \langle f_{m+1}^\prime \oplus \mu, \lambda_{m+1}(\nu),h^0,h^1,(H_{m+1}^2)^\prime(\nu)\rangle ^\frown \vec{r}_{m+1} \not \in D$.
\end{center}
 
 We can use the property of $D_{\mu(\kappa_{m+1})}$ to show that $E_\mu$ is open dense. Use the induction hypothesis to find $t(\mu) \leq^* (q^m \restriction (m+1))_{\mu}$, $h^0(\mu) \leq (H_{m+1}^0)^{q^m}(\mu \restriction d_{m+1}^{q^m})$, and $h^1(\mu) \leq (H_{m+1}^1)^{q^m}(\mu \restriction d_{m+1}^{q^m})$ with the least $[OR]^{<\omega}$-element $\vec{\alpha}^\mu=\alpha^\mu_0< \dots <\alpha^\mu_{k(\mu)-1}$, in the lexicographic order, such that for every $\vec{\tau} \in \prod\limits_{\alpha \in \vec{\alpha}^\mu}A_\alpha^{t(\mu)}$, $h^0 \leq h^0(\mu)$, and $h^1 \leq h^1(\mu)$, we have $(t(\mu)+\vec{\tau},h^0,h^1) \in E_\mu$.

For each $\mu \in A_{m+1}^\prime$, define
$F_\mu:A_{\alpha_0^\mu}^{t(\mu)} \times \dots A_{\alpha_{k(\mu)-1}^\mu}^{t(\mu)} \to 2$ by $F_{\mu}(\tau_0, \dots, \tau_{k(\mu)-1})=1$ if and only if 

\begin{center}
$(t(\mu) + \langle \tau_0, \dots , \tau_{k(\mu)-1} \rangle)^\frown$\\  $ \langle f_{m+1}^\prime \oplus \mu ,\mu(\kappa_{m+1}),h^0(\mu),h^1(\mu), (H_{m+1}^2)^\prime(\mu(\kappa_{m+1})) \rangle ^\frown \vec{r}_{m+1} \in D$.
\end{center}

By Lemma \ref{rowbottom}, we have a measure one set $B_{\alpha^\mu_i}^{t(\mu)} \subseteq A_{\alpha_i^\mu}^{t(\mu)}$ for all $i<k(\mu)$ such that
$F \restriction B_{\alpha_0^\mu}^{t(\mu)} \times \dots B_{\alpha_{k(\mu)-1}^\mu}^{t(\mu)}$ is constant. Shrink the measure one sets $A_{\alpha_0^\mu}^{t(\mu)}, \dots, A_{\alpha_{k(\mu)-1}^\mu}^{t(\mu)}$ inside $t(\mu)$ to $B_{\alpha_0^\mu}^{t(\mu)}, \dots, B_{\alpha_{k(\mu)-1}^\mu}^{t(\mu)}$, respectively, that $F_\mu$ is constant on the product of those measure one sets. Restrict the collapses based on the measure one sets we just shrank. Call the resulting condition $t^*(\mu)$. By the shrinking of measure one sets in $t(\mu)$, we arranged that

\begin{enumerate}[label=(S\arabic*)]

\item \label{S1} EITHER $t^*(\mu) + \langle \tau_0, \dots, \tau_{k(\mu)-1} \rangle^\frown \langle f_{m+1}^\prime \oplus \mu,h^0(\mu),h^1(\mu),(H_{m+1}^2)^\prime(\mu(\kappa_{m+1}) \rangle^\frown \vec{r}_{m+1} \in D$ for all $\vec{\tau}$, 
\item \label{S2} OR for all $\tau$, there are no $g \leq f_{m+1}^\prime \oplus \mu$, $h^2 \leq (H_{m+1}^2)^\prime(\mu(\kappa_{m+1}))$, and $\vec{r}^\prime \leq^* \vec{r}_{m+1}$ such that $t^*(\mu) + \langle \tau_0, \dots, \tau_{k(\mu)-1} \rangle\frown \langle g,h^0(\mu),h^1(\mu),h^2 \rangle^\frown \vec{r}^\prime \in D$.
\end{enumerate}

Shrink $A_{m+1}^\prime$ further so that every $\mu$ satisfies \ref{S1}, or every $\mu$ satisfies \ref{S2}. If every $\mu$ satisfies \ref{S1}, shrink further so that there is a sequence $\vec{\alpha}_{m+1}$ such that for every $\mu \in A_{m+1}^\prime$, $\vec{\alpha}^\mu=\vec{\alpha}_{m+1}$.

Observe that $t^*(\mu) \leq^* (q^m \restriction (m+1))_\mu$, $h^0(\mu) \leq (H_{m+1}^0)^{q^m}(\mu \restriction d_{m+1}^{q^m})$, and $h^1(\mu) \leq (H_{m+1}^1)^{q^m}(\mu \restriction d_{m+1}^{q^m})$. Use Lemma \ref{integrate} to integrate these components together to form a condition $q^{m+1}$. Hence, for $\tau \in A_{m+1}^{q^{m+1}}$ with $\mu=\tau \restriction d_{m+1}^{q^m}$, and $\nu=\mu(\kappa_{m+1})=\tau(\kappa_{m+1})$, we have $(q^{m+1}+\tau) \restriction (m+1) \leq^* t^*(\mu)$,  for $n\leq m$, $f_n^{q+\tau}=f_n^{t^*(\mu)}$, $(H_{m+1}^0)^{q^{m+1}}(\tau)=h^0(\mu), (H_{m+1}^1)^{q^{m+1}}(\tau)=h^1(\mu)$, and $(H_{m+1}^2)^{q^{m+1}}(\nu)=h^2(\nu)$. This completes the construction of $q^{m+1}$.
Here is what we have: if $q^\prime$ is an extension of $q^{m+1}$ such that $\supp(q^\prime)$ is the least in the lexicographic order in $[OR]^{<\omega}$, $\max(\supp(q^\prime))=m+1$, and $q^\prime \in D$, then every extension $q^{\prime\prime}$ of $q^{m+1}$ with $\supp(q^{\prime\prime})=\supp(q^\prime)$ is in $D$. 
Now we have $q \leq^* q^m$ for all $m$.

\begin{claim} 
$q$ satisfies the strong Prikry property.
\end{claim}

\begin{claimproof}(sketch) Let $q^\prime \leq q$ with $q^\prime \in D$. Assume $q^\prime$ is not pure with the least $\supp(q^\prime)$ in the lexicographic order in $[OR]^{< \omega}$. Enumerate $\supp(q^\prime)$ in increasing order as $\alpha_0<\dots < \alpha_{k-1}$. If $\alpha_{k-1}=0$, then the proof is easy. Assume $\alpha_{k-1}=m+1$. Using the notations from the construction of $q^{m+1}$, we have that for every $\tau \in A^{q^\prime}_{m+1}$, $\tau \restriction d_{m+1}^\prime$ satisfies the property \ref{S1}, and $\vec{\alpha}_{m+1}=\langle \alpha_0, \dots,\alpha_{k-1} \rangle$. By the way we shrank $A_{m+1}^\prime$, for every $\vec{\tau} \in \prod\limits_{\alpha \in \vec{\alpha}_{m+1}} A_\alpha^q$, we have $q+ \vec{\tau} \in D$.
\end{claimproof}

\end{proof}

\section{cardinal preservation} \label{cardinalpreserve}

This section is dedicated to showing some cardinal preservation results. We analyze the generic extension. From now, assume the length of the extender sequence is limit $\eta$. Let $\eta^+=\aleph_{\gamma_\eta}$. Also, assume $\rho=\bar{\kappa}_\eta$ and $\overline{\kappa}_0=\max\{\omega,\eta\}$. Let $G$ be $\mathbb{P}$-generic over $V$.

\begin{lemma}\label{preserve}
For each $\alpha<\eta$, the cardinal $\kappa_\alpha$ is preserved in $V[G]$, $(2^{\kappa_\alpha})^{V[G]}\leq \kappa_{\alpha+1}$. As a consequence, for any limit ordinal $\beta \leq \eta$, $\overline{\kappa}_\beta$ is a singular strong limit cardinal in $V[G]$, and $\overline{\kappa}_\beta=\aleph_{\gamma_{\eta}+\beta}^{V[G]}$.
\end{lemma}

\begin{proof}
 Let $\alpha<\eta$. By a density argument, we find $p \in G$ with $\alpha,\alpha+1 \in \supp(p)$. Then $G/p$ is factored into $G_0:=G/p \restriction \mathbb{P}_{\langle h_\alpha^\gamma(f_\alpha^p(\kappa_\alpha)) : \gamma < \alpha \rangle}$, $G_1:=G/p \restriction \{\alpha\}$ and $G_2:=G/p \restriction (\mathbb{P} \setminus (\alpha+1))$, for some $\lambda_\alpha<\kappa_\alpha$. $G_0$ comes from a
  forcing that has $\lambda_\alpha$-c.c. Thus, $\kappa_\alpha$ is preserved in the extension by
  $G_0$. $G_1$ comes from a
  forcing which is equivalent to $\mathcal{A}(\kappa_\alpha,\lambda_{\alpha+1}) \times \Col(\bar{\kappa}_\alpha,<\nu) \times \Col(\nu,s_\alpha(\nu)^+) \times \Col(s_\alpha(\nu)^{+3},<\kappa_{\alpha})$ for some inaccessible $\nu$. We see that $\kappa_\alpha$ is still preserved, and in
  the extension by $G_0$ and $G_1$, $2^{\kappa_\alpha} \leq 2^{\kappa_\alpha^+}= \lambda_{\alpha+1}<\kappa_{\alpha+1}$. 
  Finally, $G_2$ comes from a forcing whose $\leq^*$ is $\kappa_\alpha^+$-closed. By the Prikry property, it does not add new $\kappa_\alpha$-sequences. Hence, in $V[G]$, $\kappa_\alpha$ is preserved, and $(2^{\kappa_\alpha})^{V[G]} \leq \kappa_{\alpha+1}$.
  
  By our analysis, if $\alpha$ is $0$ or a successor ordinal, then on each interval $(\overline{\kappa}_\alpha,\kappa_\alpha]$, there is a $\nu_\alpha \in (\overline{\kappa}_\alpha,\kappa_\alpha)$ such that $\overline{\kappa}_\alpha^+$,  $\nu_\alpha$, $s_\alpha(\nu_\alpha)^{++}$,$s_\alpha(\nu_\alpha)^{+3}$,
  and $\kappa_\alpha$ are preserved (see Lemma \ref{limitcard}), and $\nu_0$ becomes $\eta^+$.
   If $\alpha<\eta$ is limit,  then  since $\mathbb{P}_{\langle h_\alpha^\beta(\nu):\beta<\alpha \rangle}$, for some $\nu$, is $\lambda_\alpha$-c.c., we have that in the interval $(\overline{\kappa}_\alpha,\kappa_\alpha]$, $\lambda_\alpha,\lambda_\alpha^+,\kappa_\alpha$ are preserved, while the rest but $\overline{\kappa}_\alpha^+,\nu_\alpha$ are collapsed (where $s_\alpha(\nu_\alpha)^{++}=\lambda_\alpha$), hence, we are done.
\end{proof}

We know for each limit $\beta \leq \eta$, $\bar{\kappa}_\beta$ is preserved. $(\bar{\kappa}_\beta^+)^V$ is also preserved by
this forcing:

\begin{lemma}\label{limitcard}
For any limit ordinal $\beta \leq \eta, \bar{\kappa}_\beta^+{}^V$ is preserved in $V[G]$.
\end{lemma}

\begin{proof}

We only show the case $\beta = \eta$ here. The case $\beta<\eta$ is similar, together with the fact that $(\mathbb{P} \setminus \beta, \leq^*)$ is $\bar{\kappa}_\beta^+$-closed.

Suppose not. Then in $V[G]$, let $\xi=\cf{\overline{\kappa}_\eta^+}<\overline{\kappa}_\eta$. Choose $\alpha<\eta$ such that $\xi<\kappa_\alpha$. Extend $p$ so that $\alpha \in \supp(p)$. Break $p$ into $p \restriction \alpha$, $p(\alpha)$, and $p \setminus (\alpha+1)$. Since $p \restriction \alpha$ and  the Collapse parts in $p(\alpha)$ come from forcings which have $\kappa_\alpha$-.c.c., and the Cohen part of $p(\alpha)$ comes from a forcing which is $\kappa_\alpha^+$-closed. $\overline{\kappa}_\eta$ is collapsed in the forcing in which $p \setminus (\alpha+1)$ lives (which is $\mathbb{P} \setminus (\alpha+1)$).

In $V$, let $\{\dot{\gamma}_i:i<\xi\}$ be a sequence of names, forced by $ p \setminus (\alpha+1) \in \mathbb{P} \setminus (\alpha+1)$, to be a cofinal sequence in $\bar{\kappa}_\eta^+{}^V$. Build a sequence of conditions
$\{p_i : i<\xi \}$ such that $p_0=p \setminus (\alpha+1)$, $\{p_i: i<\xi \}$ is $\leq^*$-decreasing, and $p_{i+1}$ satisfies Lemma \ref{strongprikry} for $D_i=\{q \in \mathbb{P \setminus \alpha}: q$ decides the value of $\dot{\gamma}_i \}$. 

Set $r$ to be a $\leq^*$- lower bound of $\{p_i :i<\xi\}$ in $\mathbb{P} \setminus (\alpha+1)$. By Lemma \ref{strongprikry}, for each $i<\xi$, $A_i=\{\beta : \exists r^\prime \leq r, r^\prime \Vdash \dot{\gamma}_i=\check{\beta}\}$ has size at most $\kappa_{\alpha_i}$ for some $\alpha_i<\eta$. Set $\beta_i=\sup A_i$, and $\beta=\sup\limits_{i<\xi} \beta_i$. Then $r \Vdash \sup\{ \dot{\gamma}_i:i<\xi\} \leq \check{\beta}$ and $\beta<(\bar{\kappa}_\eta^+)^V$, which is a contradiction.
\end{proof}

From the series of lemmas above and some chain condition arguments,
in the generic extension $V[G]$, we get that for each limit ordinal $\beta < \eta$, $\overline{\kappa}_\beta$, $(\overline{\kappa}_\beta^+)^V$ are preserved (where $\lambda_\beta(\nu_\beta)=\lambda_\beta^p$ for some $p \in G$ with $\beta \in \supp(p)$), and $\overline{\kappa}_\eta$, $(\overline{\kappa}_\eta^+)^V$, and $(\overline{\kappa}_\eta^{++})^V$ are preserved. Furthermore, for limit $\beta \leq \eta$, $\overline{\kappa}_\beta=\aleph_{\gamma_\eta+\beta}^{V[G]}$, $(\overline{\kappa}_\beta^+)^V=(\aleph_{\gamma_\eta+\beta+1})^{V[G]}$, and if $\beta<\eta$ is limit, then $\lambda_\beta=(\aleph_{\gamma_\eta+\beta+k})^{V[G]}$, where $k=2$ or $3$, and $\lambda_\beta=\lambda_\beta^p$ for some $p \in G$ with $\beta \in \supp(p)$. Finally, $\lambda$ becomes $\aleph_{\gamma_\eta+\eta+2}^{V[G]}$. The main reason why the value $\lambda_\beta$ for limit $\beta<\eta$ is unclear is because we do not know exactly if $\nu_\beta$ is preserved, where $s_\beta(\nu_\beta)^{++}=\lambda_\beta$.

From this point, for limit $\beta<\eta$, let $\lambda_\beta=\lambda_\beta^p$ for some $p \in V[G]$ with $\beta \in \supp(p)$, and $\lambda_\eta=\lambda$. Next, we are going to verify that in $V[G]$, for limit $\beta \leq \eta$, $2^{\overline{\kappa}_\beta}=\lambda_\beta$. On one hand, $2^{\overline{\kappa}_\beta} \leq \lambda_\beta$ by a chain condition argument.

To show $\lambda_\beta \leq 2^{\overline{\kappa}_\beta}$, we will build a scale of length $\lambda_\beta$. We analyse the scales in the next section.

\section{scale analysis}\label{scaleanalysis}

The goal of this section is to build a scale $t_\gamma: \gamma \in [\overline{\kappa}_\eta,\lambda)$ on $\overline{\kappa}_\eta$, and obtain $2^{\overline{\kappa}_\eta}=\lambda$ as desired.
We also show that the situation ``reflects down", namely for $\beta<\eta$ limit, we will have a scale on $\overline{\kappa}_\beta$ of length $\lambda_\beta$ when $\lambda_\beta=\lambda_\beta^p$ for some $p \in G$ with $\beta \in \supp(p)$.
The initial segment of the scales mentioned above of lengths successors of singular cardinals will also form very good scales.

\begin{lemma}\label{scaledef}
  Let $\beta \leq \eta$ be a limit ordinal, let $q$ be a condition such that $\beta \in \supp(q)$ if $\beta<\eta$, and let $\lambda_\beta=\lambda$ if $\beta=\eta$ and $\lambda_\beta=\lambda_\beta^q$ if $\beta<\eta$. Let $\gamma \in [\overline{\kappa}_\beta,\lambda_\beta)$ and $\alpha<\beta$. Let $D$ be the collection of $p \leq q$ such that 
\begin{enumerate}
\item $\alpha \in \supp(p)$.
\item If we enumerate $\supp(p) \cap (\beta+1) \setminus \alpha$ in decreasing order as
    $\alpha_0> \dots >\alpha_{k-1}=\alpha$ such that if $\beta<\eta$, $\alpha_0=\beta$, then $\gamma \in \dom(f^p_{\alpha_0})$, the sequence of ordinals defined inductively by setting $\gamma_0 = \gamma$ and $\gamma_{i+1} = f^p_{\alpha_i}(\gamma_i)$ for as long as $\gamma_i \in \dom(f^p_{\alpha_i})$
    reaches a stage where $\gamma_{k-1}$ is defined, and $\gamma_{k-1} \in \dom(f^p_\alpha)$. 
  \end{enumerate}  
Then $D$ is open dense below $q$.
\end{lemma}

\begin{proof}

We only prove the case $\beta=\eta$.
Clearly, $D$ is open.
Let $\gamma \in [\overline{\kappa}_\eta,\lambda)$, $\alpha \in \supp(p)$, and $p \in \mathbb{P}$.
By density, assume $\alpha \in \supp(p)$.
Enumerate $\alpha \in \supp(p)$ decreasingly as $\alpha_0> \cdots > \alpha_{k-1}=\alpha$.
If $\gamma \not \in f_{\alpha_0}^p$, we can extend $p$ so that $\gamma \in f_{\alpha_0}^p$ and map $f_{\alpha_0}^p(\gamma)$ to any appropriate value and make sure that for $\beta>\alpha_0$, $\gamma \in \dom(f_\beta^p)$ (so that $p$ still forms a condition).
Now, suppose that $i$ is the least such that $\gamma_{i+1} \not \in \dom(f_{\alpha_i}^p)$.
Then extend $p$ appropriately so that $\gamma_{i+1} \in \dom(f_{\alpha_i}^p)$.
With this process, we obtain a condition in which $\gamma_{k-1}$ is defined and is in $\dom(f_\alpha^p)$.
Hence, the extended condition is in $D$, and the proof is completed.   

\end{proof}

Fix a limit ordinal $\beta \leq \eta$. In $V[G]$, if $\beta<\eta$, let $\lambda_\beta=\lambda_\beta^p$ when $p \in G$ with $\beta \in \supp(p)$, and if $\beta=\eta$, let $\lambda_\beta=\lambda$.
Note that by genericity
of $V[G]$, $\lambda_\beta$ is well-defined.

For $\gamma \in [\overline{\kappa}_\beta,\lambda_\beta)$, define a function $t_\gamma:\beta \rightarrow \bar{\kappa}_\beta$ in $V[G]$ as follows: for $\alpha<\beta$,
  find $p \in G$ with $p$ lying in the dense set from Lemma \ref{scaledef}.
  Enumerate $\supp(p) \cap (\beta+1) \setminus \alpha$ in decreasing order as $\alpha_0> \dots >\alpha_{k-1}=\alpha$.
Define $\gamma_0, \dots ,\gamma_{k-1}$ as in Lemma \ref{scaledef}, and define $t_\gamma(\alpha)=f^p_\alpha(\gamma_{k-1})$.

To check that $t_\gamma$ is well-defined, suppose $p,q \in G$ satisfy the conditions in Lemma \ref{scaledef}. Find $r \in G$ with $r \leq p,q$. Hence, $((\supp(p) \cup \supp(q)) \cap (\beta+1)) \setminus \alpha \subseteq (\supp(r) \cap (\beta+1)) \setminus \alpha$ and $\alpha_0=\beta$ if $\beta<\eta$. Assume $r \leq^* p+ \langle \mu_0, \dots, \mu_{l-1} \rangle$ and $r \leq^* q+ \langle \tau_0, \dots, \tau_{l^\prime-1} \rangle$. For simplicity, assume $\mu_i$ is an $\beta_i$-object, $\tau_j$ is an $\zeta_j$-object, $\alpha<\beta_0< \dots < \beta_{l-1}$ and $\alpha<\zeta_0< \dots <\zeta_{m-1}$. We will show that $p$ and $r$ compute the same $t_\gamma(\alpha)$-value. A similar argument will show $q$ computes the same $t_\gamma(\alpha)$ as $r$. Simplify further that $l=1$, $\mu=\mu_0$, and $\xi=\beta_0$. Enumerate $(\supp(p) \cap (\beta+1)) \setminus \alpha$ in decreasing order as $\alpha_{k-1}> \dots \alpha_n>\alpha_{n-1} > \dots> \alpha_0=\alpha$, where $\alpha_n>\xi>\alpha_{n-1}$. Then

\begin{align*}
f_{\alpha_{k-1}}^r \circ \dots f_{\alpha_n}^r \circ f_\xi^r \circ f_{\alpha_{n-1}}^r \circ \dots f_{\alpha_0}^r(\gamma)
&= f_{\alpha_{k-1}}^p \circ \dots f_{\alpha_n}^r \circ f_\xi^r \circ f_{\alpha_{n-1}}^p \circ \dots f_{\alpha_0}^p(\gamma) \\
&=f_{\alpha_{k-1}}^p \circ \dots f_{\alpha_n}^p \circ \mu^{-1} \circ \mu \circ f_{\alpha_{n-1}}^p \circ \dots f_{\alpha_0}^p(\gamma) \\
& =f_{\alpha_{k-1}}^p \circ \dots f_{\alpha_n}^p \circ f_{\alpha_{n-1}}^q \circ \dots f_{\alpha_0}^p(\gamma).
\end{align*}

Thus, $p$ and $r$ compute the same $t_\gamma(\alpha)$.
Lemmas \ref{unbdd} and \ref{scale} are parallel to Lemmas
2.29 and 3.12 in \cite{Gitikblow},

\begin{lemma}\label{unbdd} 
In $V[G]$, $\langle t_\gamma : \gamma \in[\bar{\kappa}_\beta,\lambda_\beta)\rangle $ is $<^*$- increasing, where $t<^*t^\prime$ means there is $\alpha<\eta$ such that for all $\alpha^\prime>\alpha$, $t(\alpha^\prime)<t^\prime(\alpha^\prime)$.
\end{lemma}

\begin{proof}
  We prove the case $\beta=\eta$. The case $\beta<\eta$ is similar. Let $\gamma<\gamma^\prime \in [\overline{\kappa}_\eta,\lambda)$.
    We will show the conclusion by a density argument. Let $p\in \mathbb{P}$. We can find $p^\prime \leq p$ such that $\gamma$ and $\gamma^\prime$ are in the domains of the Cohen parts of $p$. Assume $\gamma$ and $\gamma^\prime$ belong to
$\dom(f^{p^\prime}_{\alpha_0})$ where $\max(\supp(p^\prime)) < \alpha_0$

We can also assume that for $\alpha>\alpha_0$, the domain of each object in
$A_\alpha^{p^\prime}$ contains $\gamma$, and $\gamma^\prime$. We will show
\begin{center}
    $p^\prime \Vdash \forall \alpha>\alpha_0(\dot{t}_\gamma(\alpha)<\dot{t}_{\gamma^\prime}(\alpha))$.
\end{center}
This is true because for each $\alpha>\alpha_0$, we can find $q \leq p^\prime$ with $\alpha \in \supp(q)$ and for $\alpha^\prime \in \supp(q) \setminus \alpha$, we use an addable $\alpha^\prime$-object whose domain contains $\gamma$ and $\gamma^\prime$. Every addable object is order-preserving, and by a density argument, we are done.
\end{proof}

In particular, we conclude that in $V[G]$, $2^{\aleph_{\gamma_\eta+\beta}} =\aleph_{\gamma_\eta+\beta+k}$ for $\beta<\eta$, where $k=2$ or $3$, and $2^{\aleph_{\gamma_\eta+\eta}}=\aleph_{\gamma_\eta+\eta+2}$.

We set $\lambda_\alpha$ as $\lambda^p_\alpha$ when $p\in G$ and $\alpha \in \supp(p)$. We have 

\begin{lemma}\label{scale}
In $V[G]$, $\langle t_\gamma : \gamma \in[\overline{\kappa}_\beta,\lambda_\beta)\rangle $ is a scale in $(\prod\limits_{\alpha<\beta}\lambda_\alpha,<_{bd})$. 
\end{lemma}

\begin{proof}

  Again, assume for simplicity that $\beta=\eta$. First, note that for each condition $p$ and $\alpha>\max(\supp(p))$, $\ran (\mc_\alpha(d_\alpha^p))\subseteq \lambda=j(s_\alpha)(\kappa_\alpha)^{++}$. Hence, there is a measure one set of $\mu$ such that $\ran(\mu)\subseteq s_\alpha(\mu(\kappa_\alpha))^{++}=\lambda_\alpha(\mu)$. Hence, the type is correct.

  Let $\dot{h}$ be a name and $p$ be a condition forcing that $\dot{h} \in \prod\limits_{\alpha<\eta} \dot{\lambda}_\alpha$. Suppose now for simplicity that $p$ is pure. For $\alpha<\eta$, let $D_\alpha=\{q: q \mbox{ decides } \dot{h}(\alpha)\}$. Find $q \leq^* p$ witnessing the strong Prikry property for $D_\alpha$, with the finite set of coordinates
  $I_\alpha$, for all $\alpha<\eta$. Assume further that $\alpha \in I_\alpha$. Define $Y_\alpha(\vec{\mu})$ to be the value that $q+ \vec{\mu}$ decides for $\dot{h}(\alpha)$. Note that this is less than $\lambda_\alpha(\mu_\alpha)$, which is a regular cardinal. Now for $\vec{\mu} \in \prod\limits_{\delta \in I_\alpha \setminus \alpha} A^q_\delta$, define $Z_\alpha(\vec{\mu}):=\sup\limits_{\vec{\tau}} (Y(\vec{\tau}^\frown\vec{\mu}))+1$. By a simple counting argument $Z_\alpha(\vec{\mu})<\lambda_\alpha(\mu_\alpha)$, and for $\vec{\mu} \in\prod\limits_{\delta \in I_\alpha \setminus \alpha} A^q_\delta$, we still have $q+\vec{\mu} \Vdash \dot{h}(\alpha)< Z_\alpha(\vec{\mu})$. We have that for $\vec{\mu} \in \prod\limits_{\delta \in I_\alpha \setminus (\alpha+1)}j_{E_\alpha}(A_\delta^q)$,
  
 \begin{center}
  $j_{E_\alpha}(q)+(\langle \mc_\alpha(d_\alpha^q) \rangle^\frown \vec{\mu}) \Vdash j_{E_\alpha}(\dot{h}(\alpha))<\lambda$.
  \end{center}
  
   Since for $\delta>\alpha$, $j_{E_\alpha}(A_\delta^q)$ comes from a measure which is $j_{E_\alpha}(\kappa_\delta)$-complete, and $j_{E_\alpha}(\kappa_\delta)>j_{E_\alpha}(\kappa_\alpha)\geq \lambda$, there are $\gamma_\alpha<\lambda$ and measure one set $B_\delta^\alpha$ for $\delta \in I_\alpha \setminus (\alpha+1)$ such that for $\vec{\mu} \in \prod\limits_{\delta \in I_\alpha \setminus (\alpha+1)}B_\delta^\alpha$,

   \begin{center}
   $j_{E_\alpha}(q)+ (\langle \mc_\alpha(d_\alpha^q) \rangle^\frown \vec{\mu}) \Vdash j_{E_\alpha}(\dot{h}(\alpha))= \gamma_\alpha$.
   \end{center}



We run through the process as above for all $\alpha<\eta$. Take $\gamma=\sup_{\alpha<\eta} \gamma_\alpha$. Let $r \leq^* q$ with $\gamma \in \dom(f_0^r)$. Hence, for $\alpha<\eta$, and $\vec{\mu} \in \prod\limits_{\delta \in I_\alpha \setminus (\alpha+1)} j_{E_\alpha}(A_\delta^r) \cap B_\delta^\alpha$,

\begin{center}
$j_{E_\alpha}(r)+(\langle \mc_\alpha(d_\alpha^r) \rangle^\frown \vec{\mu}) \Vdash j_{E_\alpha}(\dot{h}(\alpha))<\mc_\alpha(d_\alpha^r)(j_{E_\alpha}(\gamma)).$
\end{center}

Since the sets of such $\vec{\mu}$ are of measure one, by elementarity we may shrink measure one sets $A_\delta^r$ for $\delta \in I_\alpha \setminus \alpha$ so that every extension of $r$ using objects in $\{A_\delta^r : \delta \in I_\alpha\}$ decides that $\dot{h}(\alpha)<\mu(\gamma)$ where $\mu$ is the object being used in $A_\alpha^r$. Repeat the shrinking process for all $\alpha$ and call the resulting condition $s$. 

We claim that $s \Vdash \dot{h}(\alpha)<\dot{t}_\gamma(\alpha)$ for all $\alpha$. To show this, fix an $\alpha$. Let $s^\prime \leq s$ such that $s^\prime$ decides $\dot{h}(\alpha)$. Assume $s^\prime \leq^* s^*+ \vec{\mu}$ where $\vec{\mu}$ comes from measure one sets $\{ A_\delta^r : \delta \in X\}$ and $X \supseteq I_\alpha$. By the strong Prikry property, $s^{**}:=s^* +(\vec{\mu} \restriction I_\alpha)$ decides $\dot{h}(\alpha)$ to be an ordinal less than $\mu_\alpha(\gamma)$. A straightforward  calculation tells that $s^{**}$ decides $\dot{h}(\alpha)$ to be an ordinal less than $\dot{t}_\gamma(\alpha)$. Hence, we are done.

\end{proof}

To investigate the scale further, note that if $p \in G$ and $\beta \in \supp(p)$, $\lambda_\beta=\rho_\beta^{++}$ for some $\beta$. The exact same argument shows that $\langle t_\gamma : \gamma \in [\overline{\kappa}_\beta,\rho_\beta^+)\rangle$ is a scale in $(\prod\limits_{\alpha<\beta}\rho_\alpha^+,<_{bd})$. Recall that a scale $\langle h_\alpha : \alpha<\chi^+ \rangle$ on $\prod\limits_{\beta<\theta} \theta_\beta$ is {\em very good} if modulo club filter, every $\alpha<\chi^+$ with $\cf(\alpha)>\theta$ is a {\em very good point}, meaning there is a club $C \subseteq \alpha$ of type $\cf(\alpha)$ and $\gamma<\theta$ such that for $\beta_0$ and $\beta_1$ in $C$ with $\beta_0<\beta_1$ and $\xi>\gamma$, $f_{\beta_0}(\xi)<f_{\beta_1}(\xi)$. 

\begin{lemma}
$\langle t_\gamma : \gamma \in [\overline{\kappa}_\beta,\rho_\beta^+)\rangle$ is a very good scale.
\end{lemma}

\begin{proof}
  For simplicity, assume $\beta=\eta$. Let $\gamma<\rho^+$, say $\eta<\cf(\gamma)<\kappa_\alpha$ for some $\alpha<\eta$. Let $C \subseteq (\gamma \setminus \rho$ be a club of order type $\cf(\gamma)$. Let $p \in \mathbb{P}$ be such that $\alpha+1 \in \supp(p)$. Let $\theta=\min(\eta \setminus supp(p))$. Extend $p$ to $p^\prime$ so that $C \subseteq d_\theta^{p^\prime}$. Shrink the measure one set $A_{\alpha+1}^{p^\prime}$ so that the domain of any $\theta$-object in the measure one set contains $C$. Call the final condition $q$. It is easy to see that $q \Vdash \forall \beta_0,\beta_1 \in C(\beta_0<\beta_1 \rightarrow \forall \xi>\theta(f_{\beta_0}(\xi)<f_{\beta_1}(\xi)))$. Hence, the scale is very good.
\end{proof}

\begin{rmk}\label{mainrmk}

\begin{itemize}

\item \label{merge} In  Theorem \ref{main}, it is possible to get the equation $2^{\aleph_{\gamma+\beta}}=\aleph_{\gamma+\beta}^{+n}$ for limit $\beta<\eta$.
To do so, in the forcing describe in section \ref{forcing}, the third and the fourth coordinates will be merged.
For example, we consider $n=2$, in Definition \ref{poset} item (\ref{impurecoll}), instead of having $h_\alpha^0$ and $h_\alpha^1$, we merge to only $h_\alpha \in \Col(\overline{\kappa}_\alpha^+,\rho_\alpha^+)$, and so on.
Then in $V[G]$, for $\beta<\eta$, the value $2^{\aleph_{\gamma+\beta}}$ can be specified, and it will be $\aleph_{\gamma+\beta+2}$.

\item The noise $\gamma$ in $\aleph_{\gamma+\beta}$ and $\aleph_{\gamma+\eta}$ can be removed for several possible $\eta$.
The cause of the noise is from the fact that to prove the Prikry property, we have to make sure that the very first collapse appearing in Definition \ref{poset} is $\Col(\eta^+,<\nu)$ for some $\nu$, and if $\eta$ can be large, then cardinals below $\eta$ remains untouched.
We exemplify as follows.
First, we allow $\eta$, the length of the sequence of strong cardinals, to be any value, the requirement $\eta<\kappa_0$ is not necessary.
If $\eta \geq (\overline{\kappa}_\omega)^{+n}$ is regular, then first force by $\mathbb{P}_0:=\mathbb{P}_{\langle E_\alpha:\alpha<\omega \rangle}$ when $\lambda=(\sup_{n<\omega} \kappa_n)^{+n}$.
Since $\mathbb{P}_0$ has size $\lambda<\kappa_\omega$. 
Hence, in $V^{\mathbb{P}_0}$, $\langle \kappa_\alpha: \alpha \in [\omega,\eta) \rangle$ is still a sequence of strong cardinals, and $\eta$ is still regular.
We can proceed this ``chopping" process as we like.
If $\eta<\bar{\kappa}_\omega$, then modify the forcing to make sure that $\eta$ is preserved.

\end{itemize}

\end{rmk}

\section{acknowledgement}
We are extremely grateful to James Cummings for his guidance and encouragement in completing this project. We also thank Moti Gitik and Omer Ben-Neria for explaining some points. We would like to thank the referees for the comments and suggestions. 

\bibliographystyle{ieeetr}
\bibliography{references}

\textsc{School of Mathematical Sciences, Tel Aviv University, Tel Aviv-Yafo,  Tel Aviv, Israel, 6997801}
\newline \textit{E-mail address}: \texttt{sittinon@mail.tau.ac.il}
\newline \textit{URL}: \texttt{\url{http://www.math.tau.ac.il/~sittinon/}}
    
\end{document}